\documentclass[11pt,oneside,reqno]{amsart}
\usepackage[latin9]{inputenc}
\usepackage{geometry}
\usepackage{bm}
\usepackage{amsbsy}
\usepackage{amstext}
\usepackage{amsthm}
\usepackage{amssymb}
\usepackage{graphicx}
\usepackage{setspace}
\makeatletter

\newcommand{\lyxdot}{.}

\numberwithin{equation}{section}
\numberwithin{figure}{section}
\theoremstyle{plain}
\newtheorem{thm}{\protect\theoremname}[section]
\theoremstyle{remark}
\newtheorem{notation}[thm]{\protect\notationname}
\theoremstyle{plain}
\newtheorem{prop}[thm]{\protect\propositionname}
\theoremstyle{definition}
\newtheorem{defn}[thm]{\protect\definitionname}
\theoremstyle{plain}
\newtheorem{lem}[thm]{\protect\lemmaname}
\theoremstyle{remark}
\newtheorem{rem}[thm]{\protect\remarkname}

\usepackage{amsthm}\usepackage{mathrsfs}\usepackage{nameref}\usepackage{cite}\usepackage{upgreek}
\usepackage{stmaryrd}
\usepackage{etoolbox}
\usepackage{inputenc}

\DeclareFontFamily{U}{matha}{\hyphenchar\font45}
\DeclareFontShape{U}{matha}{m}{n}{
      <5> <6> <7> <8> <9> <10> gen * matha
      <10.95> matha10 <12> <14.4> <17.28> <20.74> <24.88> matha12
      }{}
\DeclareSymbolFont{matha}{U}{matha}{m}{n}
\DeclareFontSubstitution{U}{matha}{m}{n}

\DeclareFontFamily{U}{mathx}{\hyphenchar\font45}
\DeclareFontShape{U}{mathx}{m}{n}{
      <5> <6> <7> <8> <9> <10>
      <10.95> <12> <14.4> <17.28> <20.74> <24.88>
      mathx10
      }{}
\DeclareSymbolFont{mathx}{U}{mathx}{m}{n}
\DeclareFontSubstitution{U}{mathx}{m}{n}

\DeclareMathDelimiter{\vvvert}{0}{matha}{"7E}{mathx}{"17}

\numberwithin{equation}{section}

\def\th@plain{\thm@notefont{}\itshape}
\def\th@definition{\thm@notefont{}\normalfont}

\makeatother

\providecommand{\definitionname}{Definition}
\providecommand{\lemmaname}{Lemma}
\providecommand{\notationname}{Notation}
\providecommand{\propositionname}{Proposition}
\providecommand{\remarkname}{Remark}
\providecommand{\theoremname}{Theorem}

\begin{document}
\title[Two-component CWP Model with Three Spins]{Energy Landscape of the Two-component Curie--Weiss--Potts Model
with Three Spins}
\author{Daecheol Kim}
\address{Department of Mathematical Sciences, Seoul National University, Republic
of Korea.}
\email{dachi93@snu.ac.kr}
\begin{abstract}
In this paper, we investigate the energy landscape of the two-component
spin systems, known as the Curie--Weiss--Potts model, which is a
generalization of the Curie--Weiss model consisting of $q\ge3$ spins.
In the energy landscape of a multi-component model, the most important
element is the relative strength between the inter-component interaction
strength and the component-wise interaction strength. If the inter-component
interaction is stronger than the component-wise interaction, we can
expect all the components to be synchronized in the course of metastable
transition. However, if the inter-component interaction is relatively
weaker, then the components will be desynchronized in the course of
metastable transition. For the two-component Curie--Weiss model,
the phase transition from synchronization to desynchronization has
been precisely characterized in studies owing to its mean-field nature.
The purpose of this paper is to extend this result to the Curie--Weiss--Potts
model with three spins. We observe that the nature of the phase transition
for the three-spin case is entirely different from the two-spin case
of the Curie--Weiss model, and the proof as well as the resulting
phase diagram is fundamentally different and exceedingly complicated.
\end{abstract}

\maketitle
\tableofcontents

\section{\label{sec1}Introduction}

The Ising model is a ferromagnetic spin system consisting of two spins
and plays a vital role in the mathematical study of stochastic interacting
systems. In particular, its rich phase transition behaviors have been
rigorously studied over the last century. One of the simplest Ising
model is the Curie--Weiss model, which is an Ising model defined
on a complete graph. Owing to its mean-field feature, it is possible
to completely characterize the energy landscape of the Curie--Weiss
model; essentially, everything is known for this model. We refer the
reader to a monograph \cite{Timo} for a comprehensive discussion
of these classic results.

The Potts model is a ferromagnetic spin system consisting of $q\ge3$
spins and hence, can be regarded as a simple extension of the Ising
model. A surprising fact is that in most cases, the Potts model exhibits
qualitatively different and more complex behaviors than the Ising
model. For example, the characterization of the critical temperature
for the Potts model on a lattice requires a non-trivial argument and
is obtained in \cite{mm2}, and the phase transition at this temperature
is known to be continuous as in the Ising model if $q\le4$ (cf. \cite{1<q<4}).
For $q>4$, the phase transition is discontinuous (cf. \cite{q>4}). 

The Potts model on a complete graph is called the Curie--Weiss--Potts
model, and it has been investigated in \cite{Potts 1,Potts 2,Potts 3}.
Even if it is a mean-field model, allowing us to perform dimension
reduction on $q$ variables representing the empirical magnetization
vector, the analysis of the energy landscape is not a simple task.
The phase transition of the scaling limit of the mean-field free energy
of the Curie--Weiss--Potts model has been studied in \cite{q -> infty}.
The energy landscape was first studied in \cite{Ellis} from the viewpoint
of equivalent and nonequivalent ensembles of phase transitions. For
$q=3$, the complete analysis of the energy landscape is carried out
in \cite{q=00003D3,Lan - Seo CWP q=00003D3} and the metastability
of the Glauber dynamics has been studied based on it in \cite{Lan - Seo CWP q=00003D3}.
The cut-off phenomenon for the high-temperature regime has been proved
in \cite{cutoff} for all $q\ge3$. Extending the computation for
a model with $q\ge4$ is far more complicated and has been carried
out recently in \cite{J.Lee CWP}. It has been observed in that study
that the structure of the energy landscape for the model with $q\ge5$
differs from that of the model with $q=3,\,4$.

Recently, investigations on multi-component Curie--Weiss or Curie--Weiss--Potts
models have garnered much interest. These models are defined on a
complete graph; however, the vertices are divided into several groups
of macroscopically non-negligible sizes . The interaction strength
between two sites depends on the groups to which these sites belong.
Then, the empirical magnetization for each component interacts with
the other components, and therefore rich behaviors according to the
temperature and interaction parameters are expected. This has been
investigated in recent literature. For example, \cite{LDP =000026 phase block Potts}
presented the law of large numbers and large deviation principle for
the empirical magnetizations for the multi-component Curie--Weiss--Potts
model. The corresponding central limit theorem and moderate deviation
principle have been investigated in \cite{CLT block Potts}, and the
central limit theorem for the joint distribution for empirical magnetization
in the two-component Curie--Weiss--Potts model has been analyzed
in \cite{CLT bipart Potts} by inspecting the limiting behavior of
the free energy. As observed in these studies, the Curie--Weiss--Potts
model maintains the mean-field feature; however, the structure of
the energy landscape might be qualitatively different from that of
the one-component model provided that the inter-component interaction
parameters are considerably different from the component-wise interaction.
The present study seeks to quantitatively characterize the borderline
for this behavior in terms of the relative interaction strengths between
the component, and this is done by analyzing the energy landscape
of the simplest possible multi-component Curie--Weiss--Potts model,
namely the two-component model with three spins. We will verify that
even in this simplest possible case, with a mean-field structure,
the energy landscape is unbelievably complicated and therefore its
complete characterization is a highly demanding task. We remark that
this investigation has been carried out for the two-component Curie--Weiss
model in \cite{Bipartite q=00003D2}, but our work shows that not
only the proof but also the results for the Curie--Weiss--Potts
model are fundamentally different.

\section{\label{sec2}Main Results}

\subsection{Two-component Curie--Weiss--Potts Models}

We first define the two-component Curie--Weiss--Potts model. For
the simplicity of notation, we assume that two components $\Lambda_{N}^{(1)}=\{1,\dots,N\}$
and $\Lambda_{N}^{(2)}=\{1,\dots,N\}$ are of the same size. We consider
a spin system on $\Lambda_{N}=\Lambda_{N}^{(1)}\cup\Lambda_{N}^{(2)}$.We
denote the set of spins by $\mathcal{S}=\{1,\dots,q\}$ and define
\[
\Omega=\mathcal{S}^{\Lambda_{N}^{(1)}}\times\mathcal{S}^{\Lambda_{N}^{(2)}}
\]
to be the set of spin configurations. Each configuration $\boldsymbol{\sigma}$
in $\Omega$ is denoted by $\boldsymbol{\sigma}=(\boldsymbol{\sigma}^{(1)},\,\boldsymbol{\sigma}^{(2)})$
where 
\[
\boldsymbol{\sigma}^{(k)}=(\sigma_{1}^{(k)},\,\dots,\,\sigma_{N}^{(k)})\;\;\;\;;\;k=1,\,2,
\]
and where $\sigma_{i}^{(k)}$ denotes the spin assigned at site $i$
of the $k$-th component $\Lambda_{N}^{(k)}$. 

Now we define the Curie--Weiss--Potts (CWP) measure on $\Omega$.
We fix $J_{11},\,J_{22},\,J_{\textup{12}}\ge0$ and define the CWP
Hamiltonian $\mathbb{H}_{N}:\Omega\rightarrow\mathbb{R}$ by 
\begin{equation}
\mathbb{H}_{N}(\boldsymbol{\sigma})=-\frac{1}{N}\sum_{k=1,\,2}\sum_{1\leq i<j\leq N}J_{kk}\cdot\mathbf{1}\{\sigma_{i}^{(k)}=\sigma_{j}^{(k)}\}-\frac{1}{N}\sum_{i,\,j=1}^{N}J_{\textup{12}}\cdot\mathbf{1}\{\sigma_{i}^{(1)}=\sigma_{j}^{(2)}\}.\label{eq:ham}
\end{equation}
We denote the Gibbs measure on $\Omega$ associated with the Hamiltonian
$\mathbb{H}_{N}$ at the inverse temperature $\beta>0$ by $\mu_{N,\;\beta}$:
\[
\mu_{N,\;\beta}(\bm{\sigma})=\frac{1}{Z_{N,\;\beta}}e^{-\beta\mathbb{H}_{N}(\bm{\sigma})},
\]
where $Z_{N,\;\beta}$ is the partition function defined by
\[
Z_{N,\;\beta}=\sum_{\bm{\sigma}\in\Omega}e^{-\beta\mathbb{H}_{N}(\bm{\sigma})}.
\]
The measure $\mu_{N,\,\beta}$ is called the CWP measure, and its
energy landscape is what we seek to determine.

In order to simplify the question of the phase transition according
to the relative interaction strength, we assume that
\[
J_{11}=J_{22}=\frac{1}{1+J}\;\;\;\text{and}\;\;\;J_{12}=\frac{J}{1+J}
\]
for some $J>0$, where $J$ denotes the relative strength of the inter-component
interaction with respect to the component-wise interaction. With this
notation, \eqref{eq:ham} can be rewritten as
\begin{equation}
\mathbb{H}_{N}(\bm{\sigma})=-\frac{1}{N}\sum_{k=1,\,2}\sum_{1\leq i<j\leq N}\frac{1}{1+J}\cdot\mathbf{1}\{\sigma_{i}^{(k)}=\sigma_{j}^{(k)}\}-\frac{1}{N}\sum_{i,\,j=1}^{N}\frac{J}{1+J}\cdot\mathbf{1}\{\sigma_{i}^{(1)}=\sigma_{j}^{(2)}\}.\label{eq:ham - 2}
\end{equation}

\subsection{Distribution of empirical magnetization}

For $\boldsymbol{\sigma}\in\Omega$, we denote by\footnote{For notational simplicity, we do not stress the dependency of this
object to $N$; the same convention will be used throughout this paper.} $\boldsymbol{r}^{(k)}(\bm{\mathbb{\sigma}})=(r_{1}^{(k)}(\bm{\mathbb{\sigma}}),\,\dots,\,r_{q}^{(k)}(\boldsymbol{\sigma}))$,
$k=1,\,2$, the empirical magnetization of the component $\Lambda_{N}^{(k)}$,
i.e.,
\[
r_{u}^{(k)}(\bm{\sigma})=\frac{1}{N}\sum_{i=1}^{N}\mathbf{1}\{\sigma_{i}^{(k)}=u\}\;\;\;;\;u\in\mathcal{S}.
\]
We write $\bm{r}(\bm{\sigma})=(\boldsymbol{r}^{(1)}(\bm{\mathbb{\sigma}}),\,\boldsymbol{r}^{(2)}(\bm{\mathbb{\sigma}}))$.
Our main concern in this paper is the distribution of the empirical
magnetization $\bm{r}(\bm{\sigma})$ under the CWP measure $\mu_{N,\,\beta}$
and we shall analyze its qualitative and quantitative behavior as
$J$ and $\beta$ vary. 

The main feature of the mean-field type model as in the CWP model
is the dimension reduction. To explain this in more detail, let $\Xi$
be the $(q-1)$-dimensional simplex defined by 
\[
\Xi=\left\{ \boldsymbol{x}=(x_{u})_{1\le u\le q}\in[0,\,1]^{q}:\sum_{u=1}^{q}x_{u}=1\right\} 
\]
and let $\Xi_{N}=\Xi\cap(N^{-1}\mathbb{Z})^{q}.$ Then, we have $\bm{r}(\bm{\sigma})\in\Xi_{N}^{2}\subset\Xi^{2}$
for all $\boldsymbol{\sigma}\in\Omega$. 
\begin{notation}
\label{nota:Notation}Since $r_{q}^{(k)}(\boldsymbol{\sigma})=1-\sum_{u=1}^{q-1}r_{u}^{(k)}(\boldsymbol{\sigma})$,
there is no risk of confusion in regarding 
\[
\boldsymbol{r}^{(k)}(\bm{\mathbb{\sigma}})=(r_{1}^{(k)}(\bm{\mathbb{\sigma}}),\,\dots,\,r_{q-1}^{(k)}(\boldsymbol{\sigma})).
\]
Similarly, we can regard 
\[
\Xi=\left\{ \boldsymbol{x}=(x_{u})_{1\le u\le q-1}\in[0,\,1]^{q-1}:\sum_{u=1}^{q-1}x_{u}\le1\right\} .
\]
Hence, we shall use these two alternative expressions of the coordinate
vectors interchangeably depending on the context.
\end{notation}

Owing to the symmetry among the sites of the same component, we can
reduce the complicated spin spaces $\Omega$ into $\Xi_{N}^{2}$ by
looking at the empirical magnetization. To rigorously explain this,
we denote by $\nu_{N,\,\beta}(\cdot)$ the measure on $\Xi_{N}^{2}$
representing the distribution of empirical magnetization $\bm{r}(\bm{\sigma})$
under $\mu_{N,\,\beta}$, i.e., for $\boldsymbol{x}=(\boldsymbol{x}^{(1)},\,\boldsymbol{x}^{(2)})\in\Xi_{N}^{2}$,
\[
\nu_{N,\,\beta}(\boldsymbol{x})=\sum_{\bm{\sigma}\in\Omega:\bm{r}(\bm{\sigma})=\bm{x}}\mu_{N,\,\beta}(\boldsymbol{\sigma})=\sum_{\bm{\sigma}\in\Omega:\bm{r}(\bm{\sigma})=\bm{x}}\frac{1}{Z_{N,\;\beta}}e^{-\beta\mathbb{H}_{N}(\bm{\sigma})}.
\]
Then, we have the following expression for $\nu_{N,\,\beta}(\cdot)$.
\begin{prop}
\label{Prop Ham}We can write 
\[
\nu_{N,\,\beta}(\boldsymbol{x})=\frac{1}{\widehat{Z}_{N,\;\beta,\;J}}\exp\left\{ -N\frac{\beta}{1+J}\left(F_{\beta,\;J}(\bm{x})+\frac{1}{N}G_{N,\,\beta}(\boldsymbol{x})\right)\right\} 
\]
for some constant $\widehat{Z}_{N,\;\beta,\;J}>0$ where 
\[
F_{\beta,\;J}(\bm{x})=H(\bm{x})+\frac{1+J}{\beta}S(\bm{x})\text{\;\;\;and\;\;\;}G_{N,\;\beta,\;J}(\bm{x})=\frac{1+J}{2\beta}\log\left(\prod_{k=1,\;2}\prod_{j=1}^{q}x_{i}^{(k)}\right)+O\left(N^{-(q-1)}\right),
\]
where 
\[
H(\bm{x})=-\sum_{k=1,\;2}\sum_{i=1}^{q}\frac{1}{2}(x_{i}^{(k)})^{2}-J\sum_{i=1}^{q}x_{i}^{(1)}x_{i}^{(2)}\;\;\;\text{and}\;\;\;S(\bm{x})=\sum_{k=1,\;2}\sum_{i=1}^{q}x_{i}^{(k)}\log x_{i}^{(k)}.
\]
\end{prop}

We give a proof in Appendix \ref{sec:app A}. We note that $F_{\beta,\,J}$
indeed corresponds to the free-energy of the spin system and dominate
the energy landscape associated with the empirical magnetization.
The role of $G_{N,\,\beta,\;J}$ is limited to the sub-exponential
factor, and hence will be neglected in the remainder of this paper.
Our primary concern is the analysis of the function $F_{\beta,\,J}$.

\subsection{Two extreme cases}

In this section, we first discuss two extreme cases to explain the
main objective of this paper.
\begin{enumerate}
\item $J=\infty$ (i.e., $J_{11}=J_{22}=0$ and $J_{12}=1$) ; no component-wise
interaction,
\item $J=0$ (i.e., $J_{12}=0$ and $J_{11}=J_{22}=1$) ; no inter-component
interaction.
\end{enumerate}
\noindent In both cases, the calculations are similar. However, there
is a considerable difference between the two cases. To introduce these
results, we define a function $\xi:\left(0,\frac{1}{2}\right)\rightarrow[0,\infty)$
by
\begin{equation}
\xi(x):=\frac{1}{1-3x}\log\frac{1-2x}{x}.\label{eq:xi_invtem}
\end{equation}
The function $\xi(x)$ has a unique global minimum; thus, we denote
this value by $\beta_{1}>0$. We also define
\begin{equation}
\beta_{2}:=\frac{2(q-1)}{q-2}\log(q-1)=4\log2\;\;\;\text{and}\;\;\;\beta_{3}:=3,\label{eq:beta_2,3}
\end{equation}
which will be used later. Note that the\textbf{ }$\beta_{2}$ is introduced
in \cite{Ellis,beta2-2,Ellis Wang,J.Lee CWP} and $\beta_{1}\approx2.7465<\beta_{2}\approx2.7726<\beta_{3}$.
For a given $\beta>\beta_{1}$, $\beta=\xi(x)$ has two solutions
because $\xi(x)$ is convex. In this case, we denote the smaller solution
by $x_{s}=x_{s}(\beta)$ and the larger one by $x_{l}=x_{l}(\beta)$.
Then, we define sets
\begin{align*}
\mathfrak{S} & =\mathfrak{S}(\beta)=\left\{ \text{permutations of }(x_{s},x_{s},1-2x_{s})\right\} ,\\
\mathfrak{L} & =\mathfrak{L}(\beta)=\left\{ \text{permutations of }(x_{l},x_{l},1-2x_{l})\right\} .
\end{align*}
Furthermore, we define the set of this kind for the two-component
case as
\begin{align}
\mathfrak{S}^{2} & =\mathfrak{S}^{2}(\beta):=\left\{ \big(\bm{x}_{s,2},\bm{x}_{s,2}\big)\in\Xi_{N}^{2}:\bm{x}_{s,2}\in\mathfrak{S}\right\} ,\nonumber \\
\mathfrak{L}^{2} & =\mathfrak{L}^{2}(\beta):=\big\{\big(\bm{x}_{l,2},\bm{x}_{l,2}\big)\in\Xi_{N}^{2}:\bm{x}_{l,2}\in\mathfrak{L}\big\}.\label{eq:Nota critical}
\end{align}
We denote the point $\big(\bm{x}_{s,2},\bm{x}_{s,2}\big)\in\mathfrak{S}^{2}$
by $\bm{x_{s}}$, and the point $\big(\bm{x}_{l,2},\bm{x}_{l,2}\big)\in\mathfrak{L}^{2}$
by $\bm{x_{l}}.$ Now, we can define the following definitions.
\begin{defn}
The two components are said to be\textit{ synchronized} if all the
local minima and lowest saddles of $F_{\beta,\;J}$ belong to either
$\mathfrak{S}^{2}$ or $\mathfrak{L}^{2}$. Otherwise, the two components
are said to be \textit{desynchronized.}

Therefore, under the synchronization regime, the two components are
synchronized in the course of metastable transition from a local minima
to a global minima. On the other hand, the two components are desynchronized
in the course of metastable transition if we are in the desynchronization
regime. Then, we have the following results for the extreme cases.
\end{defn}

\begin{thm}[Main Theorem]
\label{thm:-extreme case} Suppose that $q=3$.
\begin{enumerate}
\item If there is no component-wise interaction in each component, then
the two components are synchronized.
\item If there is no inter-component interaction between the two components,
then the two components are desynchronized.
\end{enumerate}
\end{thm}

\subsection{Synchronization--desynchronization phase transition}

In the previous section, we defined synchronization between the components
and stated the results of the two extreme cases. Naturally, it can
be expected that synchronization occurs when $J$ is sufficiently
large. However, when $J$ is sufficiently small, one expected that
synchronization to be broken. In this section, we introduce results
on the boundaries of the synchronization when $q=2$ and $q=3$, respectively.
The Ising case (i.e., $q=2$) can be handled as in \cite{Bipartite q=00003D2}
and it can be restated as the following theorem, which is more comprehensive;
the results of the theorem are stated in Theorem \ref{thm:class of critical points q=00003D2}. 
\begin{thm}
Suppose that $q=2.$ Define a function $\zeta_{1}:[2,\infty)\rightarrow[0,\infty)$
by
\begin{equation}
\zeta_{1}(\beta):=\frac{\beta-2}{\beta+2}.\label{eq:zeta_1}
\end{equation}
 If $J>\zeta_{1}(\beta)$, then the two components are synchronized.
However, if $J<\zeta_{1}(\beta)$, then the two components are desynchronized.
\end{thm}

Nevertheless, for $q\ge3$, the synchronization--desynchronization
phase transition is very complicated to analyze. We explain this for
the case $q=3$. We define functions $\psi_{1},\;\psi_{2},\;\psi_{3}:[0,\infty)\rightarrow[0,\infty)$
by\textbf{ }
\begin{align}
\psi_{1}(\beta) & =\frac{\beta-\frac{1}{x_{l}}}{\beta+\frac{1}{x_{l}}}\cdot\mathbf{1}\{\beta\geq\beta_{3}\},\;\;\;\;\;\psi_{2}(\beta)=\frac{\beta-\frac{1}{3x_{l}(1-2x_{l})}}{\beta+\frac{1}{3x_{l}(1-2x_{l})}}\cdot\mathbf{1}\{\beta_{1}\leq\beta\leq\beta_{3}\},\;\;\;\text{and}\label{eq:psi_1,2,3}\\
\psi_{3}(\beta) & =\frac{\beta-3+\sqrt{25\beta^{2}-50\beta+1}}{2(1+6\beta)}.\nonumber 
\end{align}
A constant $J_{c}\approx0.2419$ will be specified later in Theorem
\ref{thm:P+R+S thm}. Then, we have the following main results.

\begin{figure}
\includegraphics[width=7cm,height=7cm]{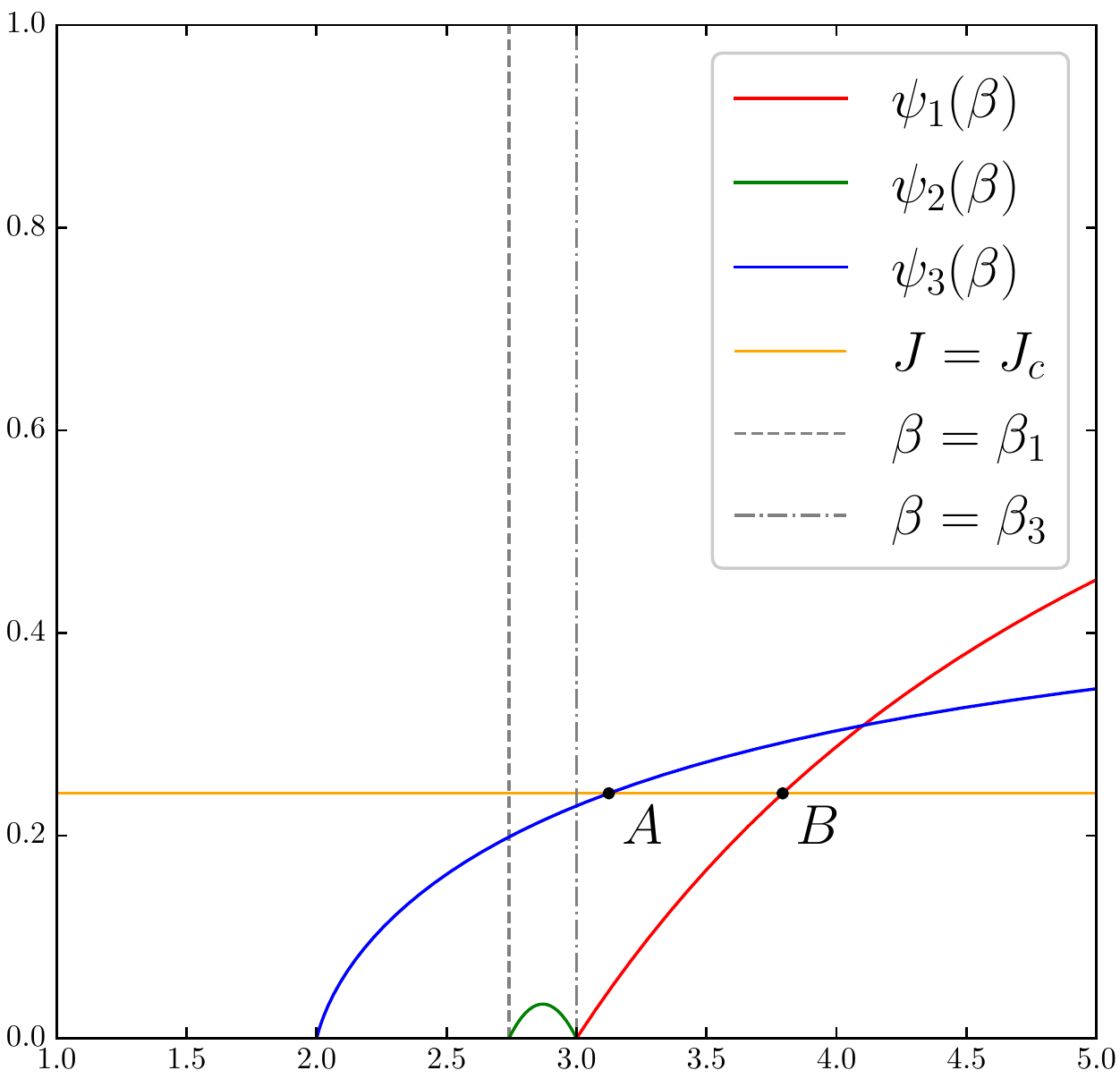}\qquad{}\includegraphics[width=7cm,height=7cm]{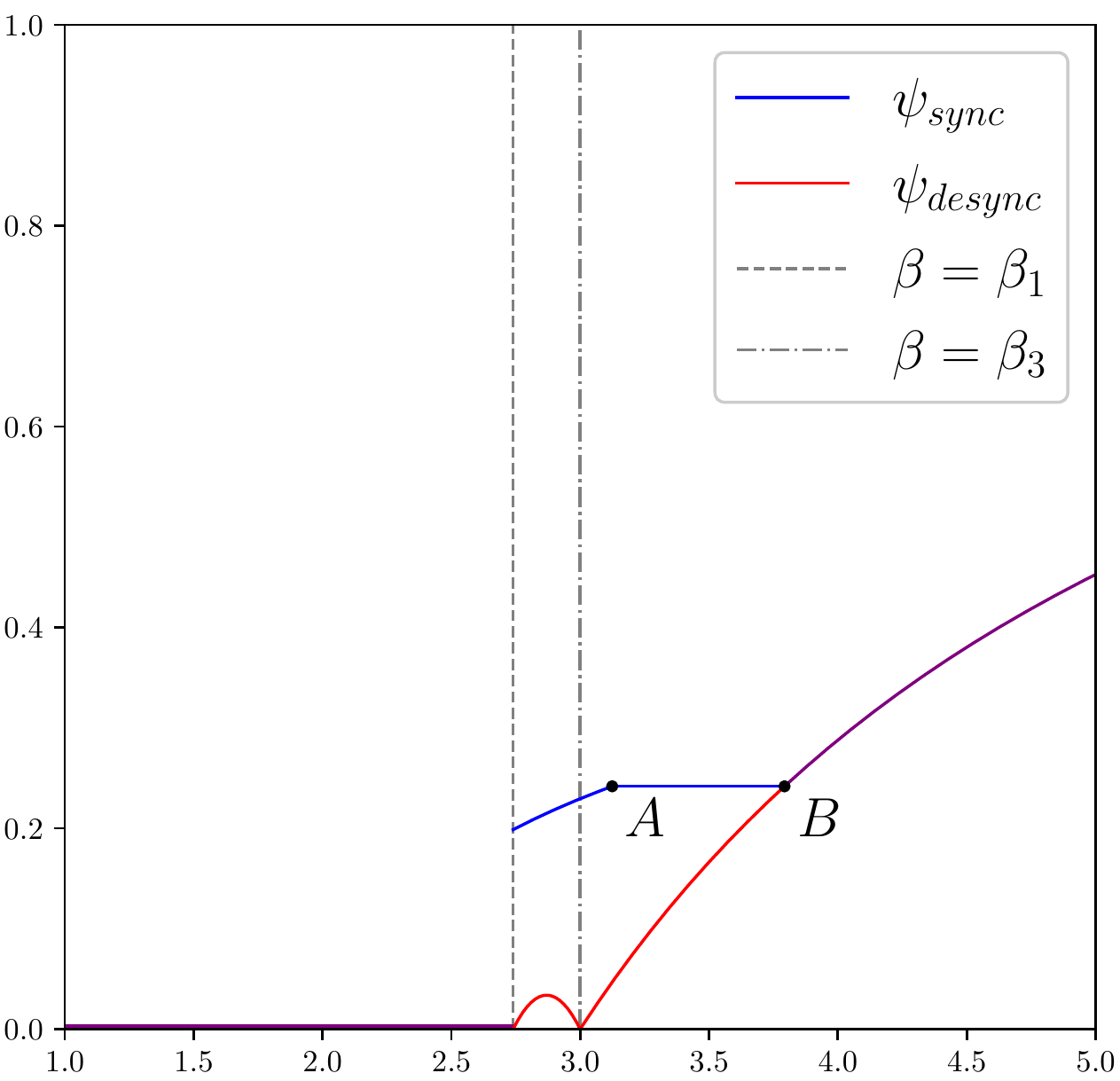}\caption{\label{fig:Phase}These are diagrams of functions of the inverse temperature
$\beta$, and the $\beta$-value of $A\approx3.1255$ and $B\approx3.8290$.
The purple curve represents the overlapping of functions $\psi_{s}$
and $\psi_{d}$. }
\end{figure}

\begin{thm}[Main Theorem]
\label{thm:sync-desync phase} We define functions $\psi_{s},\;\psi_{d}:[0,\infty)\rightarrow[0,\infty)$
by
\begin{equation}
\psi_{s}(\beta)=\begin{cases}
0 & \beta\leq\beta_{1},\\
\max[\psi_{1}(\beta),\min\{J_{c},\psi_{3}(\beta)\}] & \beta>\beta_{1},
\end{cases}\;\text{and}\;\;\;\psi_{d}(\beta)=\psi_{1}(\beta)+\psi_{2}(\beta).\label{eq:psi + and -}
\end{equation}
If $J>\psi_{s}(\beta)$, then the two components are synchronized.
However, if $J<\psi_{d}(\beta)$ , then the two components are desynchronized.
\end{thm}

We conjecture that there exists $\psi:[0,\,\infty)\rightarrow[0,\,\infty)$
such that the synchronization happens if and only if $J>\psi(\beta)$.
We proved this by verifying that $\psi(\beta)=0$ for $\beta\le\beta_{1}$
and $\psi(\beta)=\psi_{1}(\beta)$ for $\beta>\beta_{c}$ where $\beta_{c}\approx3.8290$.
This type of behavior might hold for any multi-component Curie--Weiss--Potts
model, but conditioned on the analysis and results of this paper,
it will not be easy to characterize it. This question will be pursued
future studies.

\section{\label{sec3}Ising model on bipartite graph}

In this section, we deal with the Ising case. We do this because we
present a more comprehensive proof and provide simplified versions
of the calculations when $q=3$. This emphasizes that there is a significant
difference between $q=2$ and $q=3$.

From Proposition \ref{Prop Ham}, the function $F_{\beta,\;J}$ for
$q=2$ is given by
\begin{align}
F_{\beta,\;J}(x,y) & =-\frac{{1}}{2}\left\{ x^{2}+(1-x)^{2}+y^{2}+(1-y)^{2}\}-J\{xy+(1-x)(1-y)\right\} \label{eq:F_beta J q=00003D2}\\
 & +\frac{{1+J}}{\beta}\left\{ x\log x+(1-x)\log(1-x)+y\log y+(1-y)\log(1-y)\right\} .\nonumber 
\end{align}
Here, we used the notation $x$ for $x_{1}^{(1)}$ and $y$ for $x_{1}^{(2)}$
to simplify the expression. To investigate the critical points of
$F_{\beta,\;J}(x)$, computing the partial derivatives of the function
yields
\begin{align}
 & \frac{\partial F_{\beta,\;J}}{\partial x}\left(x,y\right)=(1-2x)+J(1-2y)+\frac{{1+J}}{\beta}\log\left(\frac{{x}}{1-x}\right)=0,\label{eq:deriv q=00003D2}\\
 & \frac{\partial F_{\beta,\;J}}{\partial y}\left(x,y\right)=(1-2y)+J(1-2x)+\frac{1+J}{\beta}\log\left(\frac{{y}}{1-y}\right)=0.\nonumber 
\end{align}
An elementary computation shows that the Hessian matrix of the function
$F_{\beta,\;J}$ is
\begin{equation}
\nabla^{2}F_{\beta,\;J}(x,y)=\left(\begin{array}{cc}
\frac{{1+J}}{\beta x(1-x)}-2 & -2J\\
-2J & \frac{{1+J}}{\beta y(1-y)}-2
\end{array}\right).\label{eq:Hessian q=00003D2}
\end{equation}
Since the function $(1-2x)+J(1-2y)+\frac{{1+J}}{\beta}\log\left(\frac{{x}}{1-x}\right)$
is point symmetric at $\left(\frac{{1}}{2},\frac{1}{2}\right)$, substituting
$x=\frac{{s+1}}{2}$ and $y=\frac{{t+1}}{2}$ for $-1\leq s,t\leq1$
into equations in \eqref{eq:deriv q=00003D2} yields
\begin{equation}
\Theta(s)=t\;\;\;\text{and}\;\;\;\Theta(t)=s,\label{eq:Theta(s)=00003Dt,Theta(t)=00003Ds}
\end{equation}
where
\begin{equation}
\Theta(s)=\frac{{1}}{J}\left(-s+\frac{1+J}{\beta}\log\frac{{1+s}}{1-s}\right).\label{eq:Theta(s)}
\end{equation}

Now, we define a function $\zeta_{2},:[2,\infty)\rightarrow[0,\infty)$
by
\begin{equation}
\zeta_{2}(\beta)=\frac{\sqrt{\beta(\beta-2)}-2\log\{(\sqrt{\beta}+\sqrt{\beta-2})/\sqrt{2}\}}{\sqrt{\beta(\beta-2)}+2\log\{(\sqrt{\beta}+\sqrt{\beta-2})/\sqrt{2}\}}.\label{eq:zeta_2}
\end{equation}
Recall that $\zeta_{1}$ was defined in \eqref{eq:zeta_1}. By elementary
computations, we obtain that 
\[
\zeta_{1}(2)=\zeta_{2}(2)=0,\;\;\;\zeta_{1}(\beta)>\zeta_{2}(\beta)\;\text{for all}\;\beta>2,
\]
and
\[
\lim_{\beta\to\infty}\zeta_{1}(\beta)=\lim_{\beta\to\infty}\zeta_{2}(\beta)=1.
\]
Then, we have the following proposition.
\begin{prop}
\label{lem:G(s) property}Consider the system of equations \eqref{eq:Theta(s)=00003Dt,Theta(t)=00003Ds}.
\begin{enumerate}
\item If $\Theta'(0)>1,$ then \eqref{eq:Theta(s)=00003Dt,Theta(t)=00003Ds}
has only one solution at the origin.
\item If $-1<\Theta'(0)<1,$ then \eqref{eq:Theta(s)=00003Dt,Theta(t)=00003Ds}
has three intersections.
\item If $\Theta'(0)<-1$, then \eqref{eq:Theta(s)=00003Dt,Theta(t)=00003Ds}
has at least five intersections. Moreover, if
\begin{equation}
\Theta\left(\gamma\right)<-\gamma,\;\;\text{where}\;\;\gamma:=\sqrt{\frac{\beta-2}{\beta}},\label{eq:5 to 9points}
\end{equation}
then \eqref{eq:Theta(s)=00003Dt,Theta(t)=00003Ds} has nine intersections.
\end{enumerate}
\end{prop}

\begin{proof}
The graphs in \eqref{eq:Theta(s)=00003Dt,Theta(t)=00003Ds} is plotted
similarly to Figure $3$ in \cite{Bipartite q=00003D2} depending
on the value of $\Theta'(0)$. Note that $\Theta(s)$ is point symmetric
at (0,0) and $\Theta'(s)$ has a unique minimum at $s=0$. 

(1) and (2) are straightforward.

(3) If $\Theta'(0)<-1$, then \eqref{eq:Theta(s)=00003Dt,Theta(t)=00003Ds}
begin to intersect the line $t=-s$ at two points, except the origin.
Note that $\Theta'(\gamma)=1.$ If we denote the positive solution
of $\Theta(s)=-s$ by $s_{+}$, then the condition \eqref{eq:5 to 9points}
means that $\gamma<s_{+}$. Thus, \eqref{eq:Theta(s)=00003Dt,Theta(t)=00003Ds}
has two more intersections in the fourth quadrant, and by symmetry,
there are two more intersections in the second quadrant. This completes
the proof.
\end{proof}
From the definitions of functions $\zeta_{1}$ and $\zeta_{2}$, and
the properties of function $\Theta$, we have the following lemma.
\begin{lem}
\label{lem:equiv condition}Consider the functions $\zeta_{1},\zeta_{2},$
and $\Theta$. Then, we have the following equivalence conditions.
\begin{enumerate}
\item $\beta<2$ if and only if $\Theta'(0)>1.$
\end{enumerate}
Suppose that $\beta>2$.
\begin{enumerate}
\item[(2)]  $J>\zeta_{1}(\beta)$ if and only if $-1<\Theta'(0)<1$. However,
$J<\zeta_{1}(\beta)$ if and only if $\Theta'(0)<-1$
\item[(3)]  $J<\zeta_{2}(\beta)$ if and only if the inequality \eqref{eq:5 to 9points}
holds.
\end{enumerate}
\end{lem}

\begin{proof}
The results can be obtained from straightforward computations.
\end{proof}
\begin{lem}
\label{lem:eigenvalues q=00003D2}Suppose that $(a,b)$ is a critical
point of $F_{\beta,\;J}(x,y)$ and let $\lambda_{1},\lambda_{2}$
be the eigenvalues of the Hessian matrix $\nabla^{2}F_{\beta,\;J}(a,b)$.
Then,
\end{lem}

\begin{enumerate}
\item $\lambda_{1}+\lambda_{2}>0$ if and only if $\Theta'(2a-1)+\Theta'(2b-1)>0.$
\item $\lambda_{1}\lambda_{2}>0$ if and only if $\Theta'(2a-1)\Theta'(2b-1)>1.$
\end{enumerate}
\begin{proof}
The derivative of the $\Theta(s)$ is
\begin{equation}
\Theta'(s)=\frac{1}{J}\left(\frac{1+J}{\beta}\frac{2}{1-s^{2}}-1\right).\label{eq:Theta'(x)}
\end{equation}
By basic concepts in linear algebra and \eqref{eq:Hessian q=00003D2},
we can derive the following.

\noindent (1) The first result is calculated as
\begin{align*}
\lambda_{1}+\lambda_{2} & =\textup{tr}\left(\nabla^{2}F_{\beta,\;J}(a,b)\right)=\frac{1+J}{\beta a(1-a)}+\frac{1+J}{\beta b(1-b)}-4\\
 & =2J(\Theta'(2a-1)+\Theta'(2b-1)),
\end{align*}
(2) The second result is calculated as
\begin{align*}
\lambda_{1}\lambda_{2} & =\det\left(\nabla^{2}F_{\beta,\;J}(a,b)\right)=\left(\frac{1+J}{\beta a(1-a)}-2\right)\left(\frac{1+J}{\beta b(1-b)}-2\right)-4J^{2}\\
 & =4J^{2}(\Theta'(2a-1)\Theta'(2b-1)-1).
\end{align*}
This completes the proof.
\end{proof}
Now, we can categorize all the critical points of $F_{\beta,\;J}$
as done in \cite{Bipartite q=00003D2}. We have further found the
function \eqref{eq:zeta_2}, which is the boundary of the phase transition,
and compared the value of the local minima.
\begin{thm}[Classifications of the critical points of $F_{\beta,\;J}$]
\label{thm:class of critical points q=00003D2}$ $

\noindent The critical points of $F_{\beta,\;J}$ can be classified
as follows.
\begin{enumerate}
\item If $\beta<2$, then $F_{\beta,\;J}$ has unique local minimum at $\left(\frac{1}{2},\frac{1}{2}\right).$
\end{enumerate}
Now, suppose that $\beta>2$.
\begin{enumerate}
\item[(2)]  If $J>\zeta_{1}(\beta)$, then $F_{\beta,\;J}$ has two global minima
and a saddle point at $\left(\frac{1}{2},\frac{1}{2}\right).$
\item[(3)]  If $\zeta_{2}(\beta)<J<\zeta_{1}(\beta)$, then $F_{\beta,\;J}$
has two global minima, two saddle points, and a local maximum at $\left(\frac{1}{2},\frac{1}{2}\right)$.
\item[(4)]  If $J<\zeta_{2}(\beta)$, then $F_{\beta,\;J}$ has two global minima,
two local minima, four saddle points, and a local maximum at $\left(\frac{1}{2},\frac{1}{2}\right)$.
\end{enumerate}
\end{thm}

\begin{proof}
(1), (2) and (3) are straightforward by Proposition \ref{lem:G(s) property},
Lemma \ref{lem:equiv condition} and \ref{lem:eigenvalues q=00003D2}.

\noindent (4) By the aforementioned proposition and lemmas, there
are four local minima, four saddle points and a local maximum at $\left(\frac{1}{2},\frac{1}{2}\right)$.
Thus, it remains only to compare the values of the four local minima.
We denote the intersections of $\Theta(s)$ and $t=s$ by $(s_{1},s_{1})$
and $(-s_{1},-s_{1})$. Moreover, we denote the intersections of $\Theta(s)$
and $t=-s$ by $(s_{2},-s_{2})$ and $(-s_{2},s_{2})$. It is obvious
that $s_{1}>s_{2}$. If we set $x_{1}=\frac{1+s_{1}}{2}$ and $x_{2}=\frac{1+s_{2}}{2}$,
then by Lemma \ref{lem:eigenvalues q=00003D2}, $F_{\beta,\;J}$ has
local minima at $(x_{1},x_{1})$, $(1-x_{1},1-x_{1})$, $(x_{2},1-x_{2})$,
and $(1-x_{2},x_{2})$. By symmetry, we have
\[
F_{\beta,\;J}(x_{1},x_{1})=F_{\beta,\;J}(1-x_{1},1-x_{1})\;\;\text{and}\;\;F_{\beta,\;J}(x_{2},1-x_{2})=F_{\beta,\;J}(1-x_{2},x_{2}).
\]
In this case, we claim that
\begin{align}
F_{\beta,\;J}(x_{1},x_{1}) & <F_{\beta,\;J}(x_{2},1-x_{2}).\label{eq:size local min}
\end{align}
From the definitions of $s_{1}$ and $s_{2}$, we obtain
\begin{align}
 & Js_{1}=\frac{1+J}{\beta}\log\frac{1+s_{1}}{1-s_{1}}-s_{1}\;\;\text{and}\;\;-Js_{2}=\frac{1+J}{\beta}\log\frac{1+s_{2}}{1-s_{2}}-s_{2}.\label{eq:s_1 s_2 def}
\end{align}
An elementary calculation using the equations in \eqref{eq:s_1 s_2 def}
shows that
\[
F_{\beta,\;J}(x_{2},1-x_{2})-F_{\beta,\;J}(x_{1},x_{1})=\frac{1+J}{\beta}(f(s_{1})-f(s_{2})),
\]
where $f(x)=\frac{x}{2}\log\frac{1+x}{1-x}-2\left\{ \left(\frac{1+x}{2}\right)\log\left(\frac{1+x}{2}\right)+\left(\frac{1-x}{2}\right)\log\left(\frac{1-x}{2}\right)\right\} .$
Since the function $f(x)$ is increasing on $x>0,$ and since $s_{1}>s_{2},$
we conclude that \eqref{eq:size local min} holds. Therefore, $F_{\beta,\;J}$
has global minima at $(x_{1},x_{1})$ and $(1-x_{1},1-x_{1})$, and
has local minima at $(x_{2},1-x_{2})$ and $(1-x_{2},x_{2})$. This
completes the proof.
\end{proof}

\section{Potts model on bipartite graph with three spins}

In this section, we prove Theorem \ref{thm:sync-desync phase}, which
is the main result. In section \ref{subsec:Critical pts}, we deduce
the necessary conditions for the critical points of the function $F_{\beta,\;J}$
and calculate the eigenvalues of its corresponding Hessian matrix.
In section \ref{subsec:Extreme cases}, we prove Theorem \ref{thm:-extreme case}
by considering the two extreme cases, respectively, in Theorem \ref{thm:J=00003D0}
and Theorem \ref{thm:J=00003Dinfty}. In section \ref{subsec:Low temp},
\ref{subsec:Middle temp}, and \ref{subsec:High temp}, we find the
synchronization and desynchronization boundaries in low, medium, and
high temperatures, respectively.

\subsection{\label{subsec:Critical pts}Critical points of $F_{\beta,\;J}$}

In this section, we will find the necessary conditions for the critical
points of $F_{\beta,\;J}$ and define two functions, $\Phi$ and $\Psi$,
which are derived from them. Then, we will present the relationships
between the functions $\Phi,\Psi$ and the eigenvalues of the Hessian
matrix of $F_{\beta,\;J}$ in Lemma \ref{lem:eigenvalues}. From Proposition
\ref{Prop Ham}, the function $F_{\beta,\;J}$ for $q=3$ is given
by
\begin{equation}
F_{\beta,\;J}(\bm{x},\bm{y})=-\frac{1}{2}\sum_{i=1}^{3}(x_{i}^{2}+y_{i}^{2})-J\sum_{i=1}^{3}x_{i}y_{i}+\frac{1+J}{\beta}\sum_{i=1}^{3}(x_{i}\log x_{i}+y_{i}\log y_{i})\label{eq:F_beta,J q=00003D3}
\end{equation}
Here, we used the notations $x_{i}$ for $x_{i}^{(1)}$ and $y_{i}$
for $x_{i}^{(2)}$ to simplify the expression. To find critical points,
the first order derivatives of $F_{\beta,\;J}$ must be zero:
\begin{align*}
\frac{\partial F_{\beta,\;J}}{\partial x_{k}}(\bm{x},\bm{y})=-(x_{k}-x_{3})-J(y_{k}-y_{3})+\frac{1+J}{\beta}(\log x_{k}-\log x_{3})=0,\\
\frac{\partial F_{\beta,\;J}}{\partial y_{k}}(\bm{x},\bm{y})=-(y_{k}-y_{3})-J(x_{k}-x_{3})+\frac{1+J}{\beta}(\log y_{k}-\log y_{3})=0.
\end{align*}
for $1\leq k,l\leq3.$ Thus, we can obtain the equations
\begin{align*}
-x_{k}-Jy_{k}+\frac{1+J}{\beta}\log x_{k} & =-x_{l}-Jy_{l}+\frac{1+J}{\beta}\log x_{l},\\
-y_{k}-Jx_{k}+\frac{1+J}{\beta}\log y_{k} & =-y_{l}-Jx_{l}+\frac{1+J}{\beta}\log y_{l},
\end{align*}
for $1\leq k,l\leq3$. Since each side of the above equations are
symmetric and negative, consider the following equations:
\begin{align}
y=\Phi(x)+u\;\;\;\mbox{and}\;\;\;x=\Phi(y)+v,\label{eq:Phi(x)+u,Phi(y)+v}
\end{align}
where $u,v$ are real positive numbers and the function $\Phi$ is
defined by

\begin{equation}
\Phi(x):=\frac{1}{J}\left(-x+\frac{1+J}{\beta}\log x\right).\label{eq:Phi(x)}
\end{equation}

We need to analyze the solutions of \eqref{eq:Phi(x)+u,Phi(y)+v}
according to the values of $u$ and $v$, because they are the candidates
for the coordinates of critical points. Assume first that $u=v$.
When $u=v=0$, the graphs in \eqref{eq:Phi(x)+u,Phi(y)+v} do not
intersect since the graph $y=\Phi(x)$ is under the $x$-axis and
the graph $x=\Phi(y)$ is to the left of the $y$-axis. If we increase
$u$ gradually, $y=\Phi(x)+u$ will be tangent to $y=x$ at $x=\frac{1}{\beta}$
before intersecting with $x=\Phi(y)+u$ at two points on the line
$y=x$. We denote these points by $\bm{P}=(P,P)$ and $\bm{Q}=(Q,Q)$.
It is important to check whether these intersections are in the area
$[0,1]\times[0,1]$, otherwise these points are meaningless, because
they represent the ratio of each spin. In addition, the sum of three
of them has to be equal to one, that is, $P+2Q=1$ or $2P+Q=1$. There
is no need to consider the case when $u\neq v$ and there are two
intersections because the sum of the $x$-coordinates and the sum
of the $y$-coordinates cannot be equal to one at the same time.

Note that the smaller the $J>0$, the sharper is the graph of $y=\Phi(x)+u$.
Since the function $y=\Phi(x)+u$ is concave, \eqref{eq:Phi(x)+u,Phi(y)+v}
can have at most four intersections. We denote these intersections
by $\bm{P}=(P,P),\;\bm{R}=(R,S),\;\bm{S}=(S,R)$, and $\bm{Q}=(Q,Q)$
with $P\leq R\leq Q\leq S$. See Figure \ref{fig:Phi(x)}. The critical
points must satisfy that both the sum of their $x$-coordinates and
the sum of their $y$-coordinates should be equal to one, respectively.
For example, if we choose $\bm{P},\;\bm{R}$ and $\bm{S}$, then $P+R+S=1$.
In general, if $u\neq v$, then we denote these intersections by $\bm{P}=(P_{1},P_{2}),\;\bm{R}=(R_{1},R_{2}),\;\bm{S}=(S_{1},S_{2}),$
and $\bm{Q}=(Q_{1},Q_{2})$, with $P_{1}\leq R_{1}\leq Q_{1}\leq S_{1}.$
In this case, if we choose $\bm{P},\bm{R}$ and $\bm{S}$, then the
coordinates should satisfy $P_{1}+R_{1}+S_{1}=P_{2}+R_{2}+S_{2}=1$.

\begin{figure}
\includegraphics[width=7cm,height=7cm]{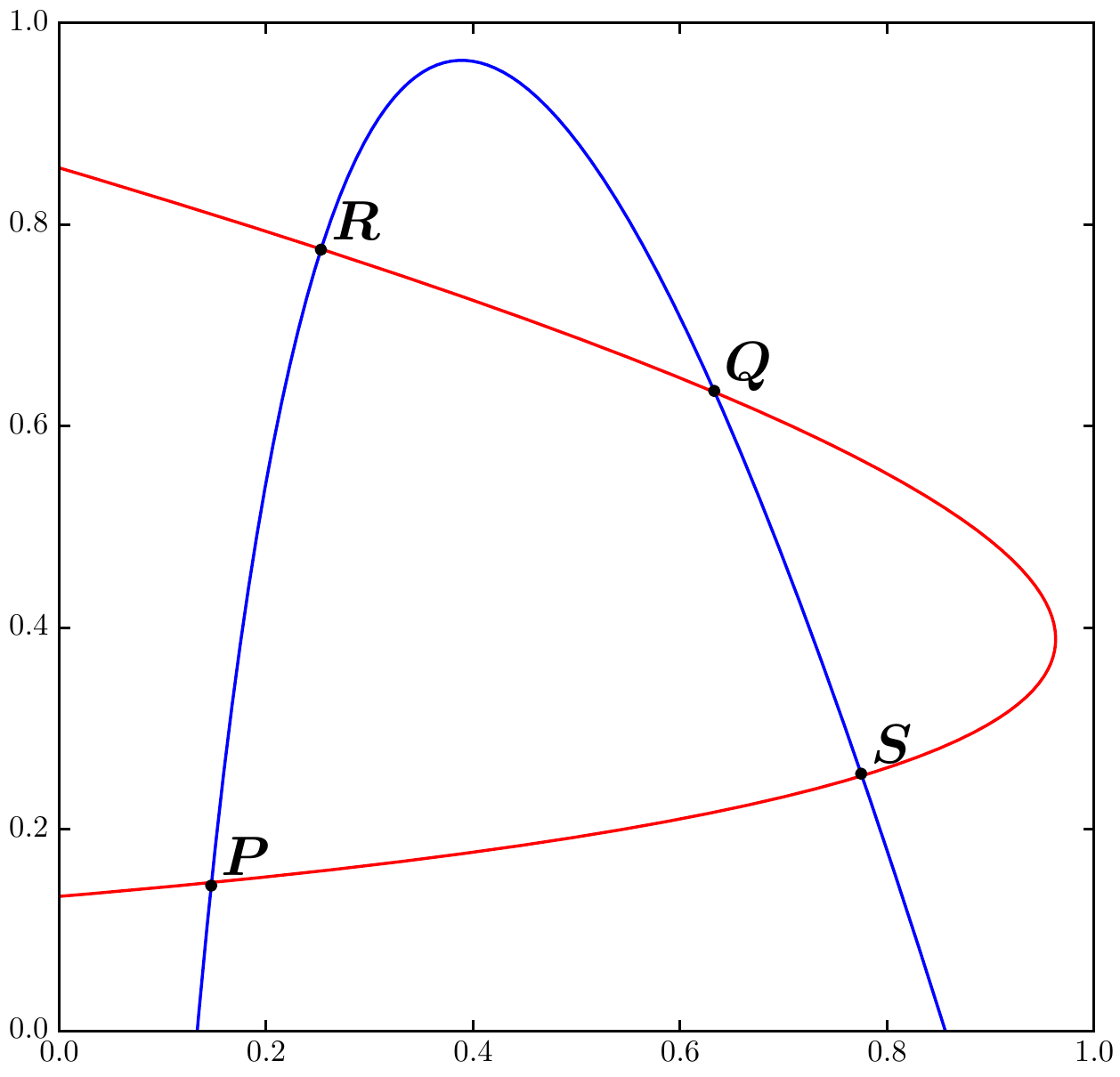}\caption{\label{fig:Phi(x)}The graphs of $y=\Phi(x)+u$ (blue curve) and $x=\Phi(y)+u$
(red curve)}
\end{figure}

\begin{rem}
The points $P_{k},Q_{k},R_{k}$ and $S_{k}$ for $k=1,2$ are functions
of the variables $u,v,J$ and $\beta$.

For now, we assume that the critical points of $F_{\beta,\;J}$ are
of the form $(s,s,1-2s,t,t,1-2t)$ up to permutations. Such critical
points satisfy the following equations:
\begin{align*}
 & y=\Phi(x)+u\;\;\;\mbox{and}\;\;\;1-2y=\Phi(1-2x)+u,\\
 & x=\Phi(y)+u\;\;\;\text{and}\;\;\;1-2x=\Phi(1-2y)+u.
\end{align*}
Subtracting the second equation from the first equation yields
\begin{align}
y & =\Psi(x)\;\;\;\text{and}\;\;\;x=\Psi(y),\label{eq:y=00003DPsi(x), x=00003DPsi(y)}
\end{align}
where
\begin{equation}
\Psi(x):=\frac{1}{3}\left(\Phi(x)-\Phi(1-2x)+1\right)=\frac{1}{3J}\left(1-3x+\frac{1+J}{\beta}\log\frac{x}{1-2x}+J\right).\label{eq:Psi(x)}
\end{equation}
Note that the function $\Psi(x)$ always passes through $(\frac{1}{3},\frac{1}{3})$.
From the equation $\Psi(x)=x$, we obtain the inverse temperature
\begin{equation}
\beta=\xi(x).\label{eq:invtem}
\end{equation}
Recall that $\xi(x)$ was defined in \eqref{eq:xi_invtem}. The derivative
of $\xi(x)$ is
\begin{equation}
\frac{d\xi(x)}{dx}=\frac{1-3x+3x(1-2x)\log\frac{x}{1-2x}}{x(2x-1)(3x-1)^{2}},\label{eq:derivative invtem}
\end{equation}
\eqref{eq:derivative invtem} has a unique solution and we denote
this by $m_{1}\approx0.2076$. Then, $\xi(x)$ has a global minimum
at $m_{1}$ and by the definition of $\beta_{1}$, we obtain 
\begin{equation}
\beta_{1}=\frac{1}{1-3m_{1}}\log\frac{1-2m_{1}}{m_{1}}.\label{eq:beta1}
\end{equation}
Recall that given $\beta>\beta_{1}$, we denoted the smaller solution
of \eqref{eq:invtem} by $x_{s}=x_{s}(\beta)$ and the larger one
by $x_{l}=x_{l}(\beta)$. Then, we can derive the relationships between
the functions $\Phi,\Psi$ and the eigenvalues of the Hessian matrix.
\end{rem}

\begin{lem}
\label{lem:eigenvalues}Suppose that the critical points of $F_{\beta,\;J}$
are of the form $(s,s,1-2s,t,t,1-2t)$ up to permutations, where $0\leq s,t\leq1$.
Let $\lambda_{1},\lambda_{2},\lambda_{3},$ and $\lambda_{4}$ be
the eigenvalues of the Hessian matrix of $F_{\beta,\;J}$. After reordering
the eigenvalues, we have
\begin{enumerate}
\item $\lambda_{1}+\lambda_{2}>0$ if and only if $\Phi'(s)+\Phi'(t)>0$
and $\lambda_{1}\lambda_{2}>0$ if and only if $\Phi'(s)\Phi'(t)>1$,
\item $\lambda_{3}+\lambda_{4}>0$ if and only if $\Psi'(s)+\Psi'(t)>0$
and $\lambda_{3}\lambda_{4}>0$ if and only if $\Psi'(s)\Psi'(t)>1$.
\end{enumerate}
\end{lem}

\begin{proof}
Note that 
\[
\Phi'(x)=\frac{1}{J}\left(\frac{1+J}{\beta x}-1\right)\;\;\;\text{and}\;\;\;\Psi'(x)=\frac{1}{J}\left(\frac{1+J}{\beta}\frac{1}{3x(1-2x)}-1\right).
\]
Before giving the proof, we first find the characteristic polynomial
of the Hessian matrix of $F_{\beta,\;J}$. Then, we compute the determinant
of the following matrix:
\begin{equation}
\nabla^{2}(F_{\beta,\;J})-\lambda I=\left(\begin{array}{cc}
\mathbf{A_{x}} & \mathbf{B}\\
\mathbf{B} & \mathbf{A_{y}}
\end{array}\right),\label{eq:eigenval}
\end{equation}
where 
\begin{align*}
\mathbf{A_{x}} & =\left(\begin{array}{cc}
\frac{1+J}{\beta}\left(\frac{1}{x_{1}}+\frac{1}{x_{3}}\right)-2-\lambda & \frac{1+J}{\beta}\frac{1}{x_{3}}-1\\
\frac{1+J}{\beta}\frac{1}{x_{3}}-1 & \frac{1+J}{\beta}\left(\frac{1}{x_{2}}+\frac{1}{x_{3}}\right)-2-\lambda
\end{array}\right),
\end{align*}
\begin{align*}
\mathbf{A_{y}} & =\left(\begin{array}{cc}
\frac{1+J}{\beta}\left(\frac{1}{y_{1}}+\frac{1}{y_{3}}\right)-2-\lambda & \frac{1+J}{\beta}\frac{1}{y_{3}}-1\\
\frac{1+J}{\beta}\frac{1}{y_{3}}-1 & \frac{1+J}{\beta}\left(\frac{1}{y_{2}}+\frac{1}{y_{3}}\right)-2-\lambda
\end{array}\right),
\end{align*}
 and $\mathbf{B}=\left(\begin{array}{cc}
-2J & -J\\
-J & -2J
\end{array}\right).$ By the assumption that $x_{1}=x_{2}=s$, $y_{1}=y_{2}=t,$ $x_{3}=1-2s$
and $y_{3}=1-2t$, $\mathbf{A_{x}}$ and $\mathbf{A_{y}}$ become
symmetric matrices. In this case, $\mathbf{BA_{y}}=\mathbf{A_{y}B}$
and hence the determinant of the Hessian matrix becomes $\det(\mathbf{A_{x}A_{y}}-\mathbf{B}^{2})$.
An elementary computation shows that
\begin{align*}
 & \det\left(\nabla^{2}\left(F_{\beta,\;J}(s,s,1-2s,t,t,1-2t)\right)-\lambda I\right)\\
 & =\{\lambda^{2}+(2-S_{1}-T_{1})\lambda+S_{1}T_{1}-S_{1}-T_{1}+1-J^{2}\}\\
 & \;\;\;\;\times\{\lambda^{2}+(6-S_{1}-T_{1}-2S_{2}-2T_{2})\lambda+S_{1}T_{1}+2S_{1}T_{2}\\
 & \;\;\;\;+2S_{2}T_{1}+4S_{2}T_{2}-3S_{1}-3T_{1}-6S_{2}-6T_{2}+9-9J^{2}\},
\end{align*}
where $S_{1}=\frac{1+J}{\beta s},\;S_{2}=\frac{1+J}{\beta(1-2s)}$
and $T_{1}=\frac{1+J}{\beta t},\;T_{2}=\frac{1+J}{\beta(1-2t)}.$
We denote the solutions of the first factor by $\lambda_{1},\lambda_{2}$
and the solutions of the second factor by $\lambda_{3},\lambda_{4}.$
Then, we may derive the following.

\noindent (1) The first result is calculated as
\begin{align*}
\lambda_{1}+\lambda_{2} & =S_{1}+T_{1}-2=J\left(\Phi'(s)+\Phi'(t)\right),\\
\lambda_{1}\lambda_{2} & =S_{1}T_{1}-S_{1}-T_{1}+1-J^{2}=J^{2}\left(\Phi'(s)\Phi'(t)-1\right),
\end{align*}
(2) The second result is calculated as
\begin{align*}
\lambda_{3}+\lambda_{4} & =S_{1}+T_{1}+2S_{2}+2T_{2}-6=3J\left(\Psi'(s)+\Psi'(t)\right),\\
\lambda_{3}\lambda_{4} & =S_{1}T_{1}+2S_{1}T_{2}+2S_{2}T_{1}+4S_{2}T_{2}-3S_{1}-3T_{1}-6S_{2}-6T_{2}+9-9J^{2},\\
 & =9J^{2}\left(\Psi'(s)\Psi'(t)-1\right).
\end{align*}
This completes the proof.
\end{proof}

\subsection{\label{subsec:Extreme cases}Two extreme cases: $J=\infty$ and $J=0$}

In this section, we prove Theorem \ref{thm:-extreme case} by categorizing
all the critical points of $F_{\beta,\;J}$ according to the value
of the inverse temperature $\beta$ in both cases $J=\infty$ (i.e.,
$J_{11}=J_{22}=0$ and $J_{12}=1$), a case without component-wise
interaction and $J=0$ (i.e., $J_{12}=0$ and $J_{11}=J_{22}=1$),
a case without inter-component interaction. For convenience, we denote
the point $\left(\frac{1}{3},\frac{1}{3},\frac{1}{3},\frac{1}{3},\frac{1}{3},\frac{1}{3}\right)$
by $\bm{q_{3}}$. First, we consider the case when $J=\infty$.
\begin{thm}
\label{thm:J=00003Dinfty}Suppose that $J=\infty$. (i.e., $J_{11}=J_{22}=0$
and $J_{12}=1$), that is, there is no component-wise interaction
in each block. Then, we have the following.
\begin{enumerate}
\item If $0<\beta<\beta_{1}$, then there is only one local minimum at $\bm{q}_{3}$.
\item If $\beta_{1}<\beta<\beta_{2}$, then there is a global minimum at
$\bm{q_{3}}$, three local minima, and three saddle points.
\item If $\beta_{2}<\beta<\beta_{3}$, then there is a local minimum at
$\bm{q_{3}}$, three global minima, and three saddle points.
\item If $\beta>\beta_{3}$, then there are three global minima and three
saddle points.
\end{enumerate}
\end{thm}

\begin{proof}
The function \eqref{eq:F_beta,J q=00003D3} becomes
\begin{equation}
F_{\beta}(\bm{x},\bm{y})=-\sum_{i=1}^{3}x_{i}y_{i}+\frac{1}{\beta}\left(\sum_{i=1}^{3}x_{i}\log x_{i}+y_{i}\log y_{i}\right).\label{eq:F_beta J=00003Dinfty}
\end{equation}
The critical points must satisfy the following equations:
\begin{align}
 & \frac{\partial F_{\beta}}{\partial x_{k}}(\bm{x},\bm{y})=-(y_{k}-y_{3})+\frac{1}{\beta}\left(\log x_{k}-\log x_{3}\right)=0,\label{eq:deri J=00003Dinfty}\\
 & \frac{\partial F_{\beta}}{\partial y_{k}}(\bm{x},\bm{y})=-(x_{k}-x_{3})+\frac{1}{\beta}\left(\log y_{k}-\log y_{3}\right)=0,\nonumber 
\end{align}
for $1\leq k\leq3.$ Since the equations in \eqref{eq:deri J=00003Dinfty}
are symmetric, we consider the following system of equations:
\begin{align}
 & y=\frac{1}{\beta}\log x+u\;\;\;\text{and}\;\;\;x=\frac{1}{\beta}\log y+v,\label{eq:logx+u logx+v}
\end{align}
where $u,v$ are real positive numbers. The function $\frac{1}{\beta}\log x$
is concave and increases; thus \eqref{eq:logx+u logx+v} has at most
two solutions. Moreover, if $x_{k}=x_{l}$, then by \eqref{eq:deri J=00003Dinfty},
we have $y_{k}=y_{l}$, for $1\leq k,l\leq3$. Thus, all the critical
points are of the form $(s,s,1-2s,t,t,1-2t)$ up to permutations.
In fact, all of them are of the form $(s,s,1-2s,s,s,1-2s)$. This
is because, from the equations
\begin{align*}
 & s=\frac{1}{\beta}\log t+u\;\;\;\text{and}\;\;\;1-2s=\frac{1}{\beta}\log(1-2t)+u,\\
 & t=\frac{1}{\beta}\log s+v\;\;\;\text{and}\;\;\;1-2t=\frac{1}{\beta}\log(1-2s)+v,
\end{align*}
we obtain
\begin{equation}
1-3s=\frac{1}{\beta}\log\frac{1-2t}{t}\;\;\;\text{and}\;\;\;1-3t=\frac{1}{\beta}\log\frac{1-2s}{s}.\label{eq:1-3s=00003Dlog1-2t}
\end{equation}
By subtracting the equations in \eqref{eq:1-3s=00003Dlog1-2t}, we
have
\[
1-3s+\frac{1}{\beta}\log\frac{1-2s}{s}=1-3t+\frac{1}{\beta}\log\frac{1-2t}{t}.
\]
Since the function $1-3x+\frac{1}{\beta}\log\frac{1-2x}{x}$ is decreasing,
we have $s=t.$ 

Thus, all the critical points are of the form $(s,s,1-2s,s,s,1-2s)$
and from \eqref{eq:1-3s=00003Dlog1-2t}, we obtain the inverse temperature
$\beta>0:$ 
\begin{equation}
\beta=\frac{1}{1-3s}\log\frac{1-2s}{s}.\label{eq:invtem J=00003Dinfty}
\end{equation}
Hence, if $\beta>\beta_{1}$, then there are two types of critical
points up to permutations, except $\bm{q_{3}}$:
\[
\bm{p_{1}}=(x_{s},x_{s},1-2x_{s},x_{s},x_{s},1-2x_{s})\;\;\text{and}\;\;\;\bm{p_{2}}=(x_{l},x_{l},1-2x_{l},x_{l},x_{l},1-2x_{l}),
\]
A straightforward calculation gives the eigenvalues of the Hessian
of \eqref{eq:F_beta J=00003Dinfty} that
\begin{equation}
\lambda_{1}=\frac{1}{\beta s}-1,\;\;\;\lambda_{2}=\frac{1}{\beta s}+1,\label{eq:eigen1,2 J=00003Dinfty}
\end{equation}
\begin{equation}
\text{\ensuremath{\lambda_{3}=}}\frac{1}{\beta s(1-2s)}-3,\;\;\text{and}\;\;\;\lambda_{4}=\frac{1}{\beta s(1-2s)}+3\label{eq:eigen 3,4 J=00003Dinfty}
\end{equation}
Since $\lambda_{2}$ and $\lambda_{4}$ are always positive, we only
need to know when $\lambda_{1}>0$ and $\lambda_{3}>0$. By substituting
\eqref{eq:invtem J=00003Dinfty} into \eqref{eq:eigen1,2 J=00003Dinfty}
and \eqref{eq:eigen 3,4 J=00003Dinfty}, we obtain that $\lambda_{1}>0$
if and only if $0<s<\frac{1}{3}$, and $\lambda_{3}>0$ if and only
if $0<s<m_{1}$ or $\frac{1}{3}<s<\frac{1}{2}$. Therefore, regardless
of the value of $\beta>\beta_{1}$, we can conclude that $\bm{p_{1}}$
is a local minimum, and $\bm{p_{2}}$ is a saddle point. The point
$\bm{q_{3}}$ is a local minimum if $\beta<\beta_{3}$, but it is
neither a local minimum nor a saddle point if $\beta>\beta_{3}$.

It remains to compare the local minima when $\beta_{1}<\beta<\beta_{2}$
and $\beta_{2}<\beta<\beta_{3}$. By symmetry, the function values
at $\bm{p_{1}}$ up to permutations are the same. A straightforward
computation yields
\[
F_{\beta,\;J}(\bm{p_{1}})-F_{\beta,\;J}(\bm{q_{3}})=\widetilde{F}(x_{s}),
\]
where
\begin{equation}
\widetilde{F}(x)=-6x^{2}+4x-\frac{2}{3}+\frac{2}{\beta}\left(2x\log x+(1-2x)\log(1-2x)+\log3\right),\label{eq:F tilde}
\end{equation}
and $\beta=\xi(x)$. The function $\widetilde{F}(x)$ is positive
on $\left(\frac{1}{6},\frac{1}{3}\right)$ and is negative on $\left(0,\frac{1}{6}\right)$
and $\left(\frac{1}{3},\frac{1}{2}\right)$. If $\beta_{1}<\beta<\beta_{2}$,
then $x_{s}\in\left(\frac{1}{6},\frac{1}{3}\right)$; hence $F_{\beta,\;J}(\bm{p_{1}})>F_{\beta,\;J}(\bm{q_{3}}).$
However, if $\beta_{2}<\beta<\beta_{3}$, then $x_{s}\in\left(0,\frac{1}{6}\right)$;
thus $F_{\beta,\;J}(\bm{p_{1}})<F_{\beta,\;J}(\bm{q_{3}})$. This
completes the proof.
\end{proof}
Next, we prove the other extreme case when $J=0$.
\begin{thm}
\label{thm:J=00003D0}Suppose that $J=0$. (i.e., $J_{11}=J_{22}=1$
and $J_{12}=0$), that is, there is no inter-component interaction
between the two blocks. Then, we have the followings.
\begin{enumerate}
\item If $0<\beta<\beta_{1}$, then there is only one local minimum at $\bm{q_{3}}.$
\item If $\beta_{1}<\beta<\beta_{2}$, then there is a global minimum at
$\bm{q_{3}}$, nine local minima and $18$ saddle points.
\item If $\beta_{2}<\beta<\beta_{3}$, then there is a local minimum at
$\bm{q_{3}}$, nine global minima and $18$ saddle points
\item If $\beta>\beta_{3},$ then there is a local maximum at $\bm{q_{3}}$,
nine global minima, and $18$ saddle points.
\end{enumerate}
\end{thm}

\begin{proof}
The function \eqref{eq:F_beta,J q=00003D3} becomes a function of
$\beta$:
\[
F_{\beta}(\bm{x},\bm{y})=-\frac{1}{2}\sum_{i=1}^{3}(x_{i}^{2}+y_{i}^{2})+\frac{1}{\beta}\sum_{i=1}^{3}(x_{i}\log x_{i}+y_{i}\log y_{i}).
\]
The critical points must satisfy the following equations:
\begin{align}
 & \frac{\partial F_{\beta}}{\partial x_{k}}(\bm{x},\bm{y})=-(x_{k}-x_{3})+\frac{1}{\beta}\left(\log x_{k}-\log x_{3}\right)=0,\label{eq:derivative J=00003D0}\\
 & \frac{\partial F_{\beta}}{\partial y_{k}}(\bm{x},\bm{y})=-(y_{k}-y_{3})+\frac{1}{\beta}\left(\log y_{k}-\log y_{3}\right)=0,\nonumber 
\end{align}
for $1\leq k\leq3.$ Clearly, $\bm{q_{3}}$ is a critical point of
$F_{\beta}$. First, we determine the form of all the critical points
of $F_{\beta}$. Since the equations in \eqref{eq:derivative J=00003D0}
are symmetric, we consider the equation
\begin{equation}
x-\frac{1}{\beta}\log x=u,\;\;\text{and}\;\;y-\frac{1}{\beta}\log y=v\label{eq:x-logx}
\end{equation}
where $u,v>0$. The function $x-\frac{1}{\beta}\log x$ is convex;
thus \eqref{eq:x-logx} each of $x$ and $y$ has at most two solutions.
Thus, all the critical points are of the form $(s,s,1-2s,t,t,1-2t)$
up to permutations and from \eqref{eq:derivative J=00003D0}, we obtain
the inverse temperature
\begin{equation}
\beta=\xi(s)=\xi(t).\label{eq:invtem J=00003D0}
\end{equation}
 The function $\xi(x)$ has a unique minimum at $x=m_{1}$ and its
value is $\beta_{1}$. Hence, there is no critical point of the form
$(s,s,1-2s,t,t,1-2t)$ when $\beta<\beta_{1}$. However, when $\beta>\beta_{1},$
from \eqref{eq:invtem J=00003D0}, $s$ can be either $x_{s}(\beta)$
or $x_{l}(\beta)$; the same is true for $t$. Hence, if $\beta>\beta_{1}$,
then there are four types of critical points up to permutations, except
$\bm{q_{3}}$:
\begin{align*}
\bm{p_{1}} & =(x_{s},x_{s},1-2x_{s},x_{s},x_{s},1-2x_{s}),\;\;\bm{p_{2}}=(x_{l},x_{l},1-2x_{l},x_{l},x_{l},1-2x_{l}),\\
\bm{p_{3}} & =(x_{s},x_{s},1-2x_{s},x_{l},x_{l},1-2x_{l}),\;\;\bm{p_{4}}=(x_{l},x_{l},1-2x_{l},x_{s},x_{s},1-2x_{s}).
\end{align*}
Since both the first and the last three coordinates of the aforementioned
critical points behave independently, there are nine permutations
for each. By substituting $J=0$ in the proof of Lemma \eqref{lem:eigenvalues},
we have
\begin{align*}
\lambda_{1}=\frac{1}{\beta s}-1,\;\;\;\lambda_{2}=\frac{1}{\beta t}-1,
\end{align*}
\begin{align*}
\lambda_{3}=\frac{1}{\beta s(1-2s)}-3,\;\;\;\text{and}\;\;\;\lambda_{4}=\frac{1}{\beta t(1-2t)}-3.
\end{align*}
From the proof of Theorem \ref{thm:J=00003Dinfty}, we know that $\lambda_{1}>0$
if and only if $0<s<\frac{1}{3}$, and $\lambda_{3}>0$ if and only
if $0<s<m_{1}$ or $\frac{1}{3}<s<\frac{1}{2}$. The same is true
for $\lambda_{2}$ and $\lambda_{4}$. Therefore, regardless of the
value of $\beta$, we can conclude that $\bm{p_{1}}$ is a local minimum,
$\bm{p_{2}}$ is neither a local minimum nor a saddle point, and $\bm{p_{3}}$
and $\bm{p_{4}}$ are saddle points. The point $\bm{q_{3}}$ is a
local minimum if $\beta<\beta_{3}$, but is a local maximum if $\beta>\beta_{3}$.
The comparison of local minima is the same as the proof of Theorem
\ref{thm:J=00003Dinfty}. This completes the proof.
\end{proof}
\begin{rem}
For any $q\ge3$, the CWP model has a first-order phase transition
at the critical temperature $\beta=\beta_{2}$, which was covered
in \cite{Ellis,beta2-2,Ellis Wang,J.Lee CWP,beta2}. We remark that
the set of global minima is replaced at $\beta=\beta_{2}$ for the
two-component case, and this critical temperature $\beta_{2}$ is
independent of variable $J$. The third statement in Theorem \ref{thm:mid temp psi_2}
cover the general case of $J$.
\end{rem}

\subsection{\label{subsec:Low temp}General case: low-temperature regime $(\beta>\beta_{3})$}

In this section, we will find all the local minima and lowest saddles
of $F_{\beta,\;J}$ which belong to either $\mathfrak{S}^{2}$ or
$\mathfrak{L}^{2}$, and the phase transition boundary for the low-temperature
part of Theorem \ref{thm:sync-desync phase}. Recall that we defined
the function $\psi_{1}$ in \eqref{eq:psi_1,2,3}. Then, the function
$\psi_{1}$ has the following properties.
\begin{prop}
The function $\psi_{1}$ is a continuous function such that
\[
0<\psi_{1}(\beta)<1,\;\;\psi_{1}(\beta_{3})=0,\;\;\text{and}\;\;\lim_{\beta\rightarrow\infty}\psi_{1}(\beta)=1.
\]
\end{prop}

\begin{proof}
Note that $x_{l}$ is a continuous function of $\beta$, and $x_{l}\in\left(\frac{1}{3},\frac{1}{2}\right)$
if $\beta>\beta_{3}$. Since $\psi_{1}(\beta)$ is a composition of
continuous functions, it is continuous. It is obvious that $\psi_{1}(\beta)<1$.
Since $x_{l}>\frac{1}{3},$ we have $\beta=\frac{1}{1-3x_{l}}\log\frac{1-2x_{l}}{x_{l}}>\frac{1}{x_{l}},$
which means that $\psi_{1}(\beta)>0$. Note that $\lim_{\beta\rightarrow\beta_{3}}x_{l}=\frac{1}{3}$
and $\lim_{\beta\rightarrow\infty}x_{l}=\frac{1}{2}$. This implies
that $\psi_{1}(\beta_{3})=\lim_{\beta\rightarrow\beta_{3}}\psi_{1}(\beta)=0$
and $\lim_{\beta\rightarrow\infty}\psi_{1}(\beta)=1$. This completes
the proof.
\end{proof}
\begin{prop}
\label{prop:psi_1 boundary}Suppose that $u=v$ in \eqref{eq:Phi(x)+u,Phi(y)+v}
and there are two intersections: $\bm{P}=(P,P)$ and $\bm{Q}=(Q,Q)$
with $P<Q$. Suppose further that $\Phi'(Q)=-1$. Then, $J>\psi_{1}(\beta)$
if and only if $P+2Q>1$. In particular, $J=\psi_{1}(\beta)$ if and
only if $P+2Q=1$.
\end{prop}

\begin{proof}
Suppose that $P+2Q>1$. Since$P<Q$, we have $Q>\frac{1}{3}$. The
condition $\Phi'(Q)=-1$ is equivalent to $Q=\frac{1+J}{\beta(1-J)}$.
By the definitions of $P$ and $Q$, we have
\begin{equation}
P-\frac{1}{\beta}\log P=Q-\frac{1}{\beta}\log Q.\label{eq:P,Q}
\end{equation}
The function $x-\frac{1}{\beta}\log x$ is convex and has a minimum
at $x=\frac{1}{\beta}$; hence $P<\frac{1}{\beta}<Q$. Since $x-\frac{1}{\beta}\log x$
is decreasing on $\left(0,\frac{1}{\beta}\right)$, from the inequality
$1-2Q<P$ and \eqref{eq:P,Q}, we obtain 
\begin{equation}
\beta<\frac{1}{1-3Q}\log\frac{1-2Q}{Q}.\label{eq:beta<Q}
\end{equation}
Thus, $Q>x_{l}=x_{l}(\beta)$, or equivalently, we obtain
\[
J>\psi_{1}(\beta)=\frac{\beta-\frac{1}{x_{l}}}{\beta+\frac{1}{x_{l}}}.
\]
On the other hand, we assume that $J>\psi_{1}(\beta),$ which is equivalent
to $Q>x_{l}$. Since $x_{l}\in\left(\frac{1}{3},\frac{1}{2}\right)$
for $\beta>\beta_{3}$, we have
\[
\psi_{1}(\beta)-\frac{\beta-3}{\beta+3}=\frac{2\beta(3-1/x_{l})}{(\beta+1/x_{l})(\beta+3)}\geq0\;\;\;\text{for all}\;\;\beta>0.
\]
Hence, $J>\frac{\beta-3}{\beta+3}$ and this implies that $Q>\frac{1}{3}$.
Since $Q>x_{l}$, \eqref{eq:beta<Q} holds. By \eqref{eq:P,Q}, the
inequality \eqref{eq:beta<Q} becomes
\[
1-2Q-\frac{1}{\beta}\log(1-2Q)<Q-\frac{1}{\beta}\log Q=P-\frac{1}{\beta}\log P.
\]
This implies that $1-2Q<P$. This completes the proof.
\end{proof}
Note that by symmetry, the function values at $\bm{x_{s}}$ or $\bm{x_{l}}$
up to permutations are the same, respectively. Then, we have the following
theorem for the low-temperature regime.
\begin{thm}
\label{thm:low temp psi_1}Suppose that $\beta>\beta_{3}$. Then,
we have the following results.
\begin{enumerate}
\item If $J>\psi_{1}(\beta)$, then $F_{\beta,\;J}$ has three local minima
at $\bm{x_{s}}\in\mathfrak{S}^{2}$ and three saddle points at $\bm{x_{l}}\in\mathfrak{L}^{2}$.
\item If $J<\psi_{1}(\beta)$, then there must be a lowest saddle of $F_{\beta,\;J}$
that does not belong to either $\mathfrak{S}^{2}$ or $\mathfrak{L}^{2}$.
\end{enumerate}
\end{thm}

\begin{proof}
(1) From the definitions of $x_{s}$ and $x_{l}$, we have
\begin{align*}
\beta & =\frac{1}{1-3x_{s}}\log\frac{1-2x_{s}}{x_{s}}=\frac{1}{1-3x_{l}}\log\frac{1-2x_{l}}{x_{l}}.
\end{align*}
Note that the function $\frac{1}{x}-\frac{1}{1-3x}\log\frac{1-2x}{x}$
is positive on $0<x<\frac{1}{3}$. Since $x_{s}<\frac{1}{3}$ and
since $x_{l}>\frac{1}{3}$, we obtain the inequality
\begin{align}
x_{s} & <\frac{1}{\beta}<x_{l}.\label{eq:abeta}
\end{align}
To analyze the critical points, we need to investigate the slopes
of $\Phi$ and $\Psi$ at $x=x_{s}$ and $x=x_{l}$. When we look
at the graph $\Psi(x)=x$, it has solutions in the order of $x_{s},\frac{1}{3}$
and $x_{l}$. It follows that $\Psi'(x_{s})>1$ and $\Psi'(x_{l})>1$.
Since $\Phi'(\frac{1}{\beta})=1$, the inequality \eqref{eq:abeta}
implies that $\Phi'(x_{s})>1$ and $\Phi'(x_{l})<1$. $F_{\beta,\;J}$
The assumption $J>\psi_{1}(\beta)$ is equivalent to
\begin{equation}
x_{l}<\frac{1+J}{\beta(1-J)}.\label{eq:-1point}
\end{equation}
The right hand side of \eqref{eq:-1point} is the point where the
function $\Phi$ has a slope of $-1$. Thus, we have $-1<\Phi'(x_{l})<1$
and by Lemma \ref{lem:eigenvalues}, $F_{\beta,\;J}$ has local minima
at $\bm{x_{s}}$ and saddle points at $\bm{x_{l}}$. This proves (1).

(2) The condition $J<\psi_{1}(\beta)$ is equivalent to $x_{l}>\frac{1+J}{\beta(1-J)}$.
This means that $\Phi'(x_{l})<-1$. Thus, by Lemma \ref{lem:eigenvalues},
$\bm{x_{l}}$ is no longer a saddle point, and hence, there is no
saddle point belonging to either $\mathfrak{S}^{2}$ or $\mathfrak{L}^{2}$.
However, according to Morse's theory, when there are two or more local
minima, there must be a saddle point with the lowest level connecting
them. Thus, there must be a saddle point that does not belong to either
$\mathfrak{S}^{2}$ or $\mathfrak{L}^{2}$. This completes the proof.
\end{proof}
\begin{lem}
\label{lem:2P+Q increase}Suppose that $u=v$ in \eqref{eq:Phi(x)+u,Phi(y)+v}
and suppose that \eqref{eq:Phi(x)+u,Phi(y)+v} intersects only two
points at $\bm{P}=(P,P)$ and $\bm{Q}=(Q,Q)$ with $P<Q$. Then $2P+Q$
and $P+2Q$ are increasing functions with respect to $u$.
\end{lem}

\begin{proof}
Note that $Q$ is increasing but $P$ is decreasing with respect to
$u$. Thus, it suffices to show that $2P+Q$ is increasing with respect
to $u$. From the definitions of $P$ and $Q$, we have
\begin{align}
 & P=\Phi(P)+u=\frac{1}{J}\left(-P+\frac{1+J}{\beta}\log P\right)+u,\nonumber \\
 & Q=\Phi(Q)+u=\frac{1}{J}\left(-Q+\frac{1+J}{\beta}\log Q\right)+u.\label{eq:def of P,Q}
\end{align}
We can replace $u$ with $\frac{1+J}{J}u$ and $P,Q$ with $\beta P,\beta Q$,
respectively. Hence, the equations in \eqref{eq:def of P,Q} become
\begin{equation}
\beta P-\log\beta P=\beta u-\log\beta\;\;\;\text{and}\;\;\;\beta Q-\log\beta Q=\beta u-\log\beta.\label{eq:def of P, Q -2}
\end{equation}
After scaling and translating the equations in \eqref{eq:def of P, Q -2},
it suffices to consider the system of equations
\begin{align}
P-\log P=u\;\;\;\mbox{and}\;\;\;Q-\log Q=u,\label{eq:ess P,Q}
\end{align}
By differentiating \eqref{eq:ess P,Q} with respect to $u$, the sum
of the derivatives of $P$ and $Q$ is
\begin{align*}
2\frac{\partial P}{\partial u} & +\frac{\partial Q}{\partial u}=\frac{2P}{P-1}+\frac{Q}{Q-1}=\frac{3PQ-2P-Q}{(P-1)(Q-1)}.
\end{align*}
Since $P$ is decreasing and $Q$ is increasing with respect to $u$,
we have $P-1<0$ and $Q-1>0$. Since $0\leq P,Q\leq1$, we have $2P+Q-3PQ\geq0.$
This completes the proof.
\end{proof}
\begin{thm}
\label{thm:P+R+S thm}Suppose that $u=v$ in \eqref{eq:Phi(x)+u,Phi(y)+v}
and suppose that the graphs in \eqref{eq:Phi(x)+u,Phi(y)+v} intersects
at four points at $\bm{P}=(P,P),\;\bm{R}=(R,S),\;\bm{S}=(S,R)$, and
$\bm{Q}=(Q,Q)$ with $P\leq R\leq Q\leq S$. Let $J_{c}$ be the positive
root of the function $(1+x)\left(\frac{1}{1-x}-\frac{1}{2+x}\right)-\log\frac{2+x}{1-x}.$
\begin{enumerate}
\item If $J\geq J_{c}$, then $P+R+S$ increases with respect to $u$.
\item The derivative $\frac{\partial(P+R+S)}{\partial u}$ increases with
respect to $u$. That is, the function $(P+R+S)(u)$ is convex.
\end{enumerate}
\end{thm}

\begin{proof}
(1) Since
\[
\Phi\left(\frac{(1+J)x}{\beta}\right)+u=\frac{1+J}{\beta}\left(\frac{1}{J}(-x+\log x)+\log\frac{1+J}{\beta}+\frac{\beta u}{1+J}\right),
\]
by scaling and translating the function $\Phi(x)$ and the variable
$u$, it suffices to consider the system of equations:
\[
\widetilde{\Phi}(P)+\frac{w}{J}=P,\;\;\widetilde{\Phi}(R)+\frac{w}{J}=S\;\;\;\text{and}\;\;\;\widetilde{\Phi}(S)+\frac{w}{J}=R,
\]
where
\begin{align}
\widetilde{\Phi}(x) & =\frac{1}{J}\left(-x+\log x\right).\label{eq:tilde f(x)}
\end{align}
An elementary computation shows that $\frac{\partial(P+R+S)}{\partial w}>0$
is equivalent to
\begin{equation}
(\widetilde{P}-J)(\widetilde{R}+\widetilde{S}+2J)-(J^{2}-\widetilde{R}\widetilde{S})>0,\label{eq:tsabc}
\end{equation}
where $\widetilde{P}:=-1+\frac{1}{P},\tilde{R}:=-1+\frac{1}{R}$ and
$\widetilde{S}:=-1+\frac{1}{S}$. We denote the geometric mean of
$R$ and $S$ by $\Gamma_{RS}:=\sqrt{RS}$ , and $\widetilde{\Gamma}_{RS}:=-1+\frac{1}{\Gamma_{RS}}.$
Then, we can rearrange \eqref{eq:tsabc} as
\begin{align}
 & (\widetilde{P}-J)\{2(J+\widetilde{\Gamma}_{RS})+(\widetilde{R}-\widetilde{\Gamma}_{RS})+(\widetilde{S}-\widetilde{\Gamma}_{RS})\}\label{eq:intermid PRS}\\
 & +\{(\widetilde{R}-\widetilde{\Gamma}_{RS})+\widetilde{\Gamma}_{RS}\}\{(\widetilde{S}-\widetilde{\Gamma}_{RS})+\widetilde{\Gamma}_{RS}\}-J^{2}>0.\nonumber 
\end{align}
Note that
\begin{align*}
\widetilde{R}-\widetilde{\Gamma}_{RS} & =\left(-1+\frac{1}{R}\right)-\left(-1+\frac{1}{\Gamma_{RS}}\right)=\frac{1}{\sqrt{R}}\left(\frac{1}{\sqrt{R}}-\frac{1}{\sqrt{S}}\right),\\
\widetilde{S}-\widetilde{\Gamma}_{RS} & =\left(-1+\frac{1}{S}\right)-\left(-1+\frac{1}{\Gamma_{RS}}\right)=\frac{1}{\sqrt{S}}\left(\frac{1}{\sqrt{S}}-\frac{1}{\sqrt{R}}\right).
\end{align*}
Substituting the aforementioned equations in \eqref{eq:intermid PRS},
we obtain
\begin{equation}
(J+\widetilde{\Gamma}_{RS})(2\widetilde{P}+\widetilde{\Gamma}_{RS}-3J)+(\widetilde{P}-J-1)\left(\frac{1}{\sqrt{R}}-\frac{1}{\sqrt{S}}\right)^{2}>0.\label{eq:apr}
\end{equation}
Then, by Lemma \ref{lem:P+R+S }, we acquire the desired results.

(2) We claim that \eqref{eq:apr} increases with respect to $u$.
Since $R$ and $S$ are getting farther as $u$ increases, it suffices
to prove that both $\widetilde{P}$ and $\widetilde{\Gamma}_{RS}$
are increasing. $\widetilde{P}$ is increasing since $P$ is decreasing.
If we show that $RS$ is decreasing, then $\widetilde{\Gamma}_{RS}$
increases with respect to $u$, achieving the required result. An
elementary computation shows that
\begin{align*}
\frac{\partial(RS)}{\partial u} & =\frac{\partial R}{\partial u}S+\frac{\partial S}{\partial u}R=\frac{1}{J^{2}-\widetilde{R}\widetilde{S}}\left(2-(1-J)(R+S)\right).
\end{align*}
By considering the graphs in \eqref{eq:Phi(x)+u,Phi(y)+v}, we observe
$R+S$ increasing with respect to $u$. Since $R+S$ has a unique
minimum at $R=S=Q=\frac{1}{1-J},$ we have $\frac{\partial(RS)}{\partial u}<0$.
This completes the proof.
\end{proof}
\begin{lem}
\label{lem:P+R+S }Under the notations introduced in Theorem \ref{thm:P+R+S thm},
we have
\begin{enumerate}
\item $J+\widetilde{\Gamma}_{RS}\geq0$,
\item $\widetilde{P}>3J$,
\item $\widetilde{P}-J-1\geq0$ if $J\geq J_{c}$.
\end{enumerate}
\end{lem}

\begin{proof}
(1) From the definitions of $R$ and $S$, we have
\[
S=\widetilde{\Phi}(R)+\frac{w}{J}\;\;\;\text{and}\;\;\;R=\widetilde{\Phi}(S)+\frac{w}{J}.
\]
Subtracting the first equation from the second equation yields
\[
(1-J)(S-R)=\log S-\log R=\int_{R}^{S}\frac{1}{t}dt.
\]
Since the sum of the area of the trapezoid with vertices $(R,0),(\Gamma_{RS},0),\big(R,\frac{1}{R}\big),\big(\Gamma_{RS},\frac{1}{\Gamma_{RS}}\big)$,
and the other one with vertices $(\Gamma_{RS},0),(S,0),\big(\Gamma_{RS},\frac{1}{\Gamma_{RS}}\big),\big(S,\frac{1}{S}\big)$
is greater than the definite integral $\int_{R}^{S}\frac{1}{t}dt$
, we obtain
\begin{align*}
(1-J)(S-R)=\int_{R}^{S}\frac{1}{t}dt & \leq\frac{1}{2}\left(\frac{1}{R}+\frac{1}{\Gamma_{RS}}\right)(\Gamma_{RS}-R)+\frac{1}{2}\left(\frac{1}{\Gamma_{RS}}+\frac{1}{S}\right)(S-\Gamma_{RS})\\
 & =\frac{1}{\Gamma_{RS}}(S-R).
\end{align*}
This proves (1).

(2) Similar to aforementioned proof, by comparing the area of the
trapezoid and the definite integral, we have
\[
(1+J)(Q-P)=\log Q-\log P=\int_{P}^{Q}\frac{1}{t}dt\leq\frac{1}{2}\left(\frac{1}{P}+\frac{1}{Q}\right)(Q-P).
\]
Thus, $\widetilde{P}+\widetilde{Q}\geq2J.$ However, $\widetilde{Q}/J=\widetilde{\Phi}'(Q)<-1$,
which is equivalent to $\widetilde{Q}<-J$. This proves (2).

(3) Define a function $h:[0,\infty)\rightarrow\mathbb{R}$ by $h(x)=(1+J)x-\log x$.
By the definitions of $P$ and $Q$, we have $h(P)=h(Q)$. The condition
$A-J-1\geq0$ is equivalent to $P\leq\frac{1}{2+J}.$ Note that the
function $h(x)$ is decreasing on $\left[0,\frac{1}{1+J}\right)$
and is increasing on $\left[\frac{1}{1+J},\infty\right)$. Thus, we
need to show that
\[
h(P)\geq h\left(\frac{1}{2+J}\right).
\]
 We know that $Q\geq\frac{1}{1-J}$; hence, we obtain $h(P)=h(Q)\geq h\big(\frac{1}{1-J}\big)$.
Consider the following function:
\begin{equation}
h\left(\frac{1}{1-J}\right)-h\left(\frac{1}{2+J}\right)=(1+J)\left(\frac{1}{1-J}-\frac{1}{2+J}\right)-\log\frac{2+J}{1-J}.\label{eq:J_c}
\end{equation}
The right-hand side of \eqref{eq:J_c} is nonnegative if $J\geq J_{c}$,
where $J_{c}$ is the positive root of \eqref{eq:J_c}. This implies
that $h(P)=h(Q)\geq h\left(\frac{1}{1-J}\right)\geq h\left(\frac{1}{2+J}\right)$
if $J\geq J_{c}.$ This completes the proof.
\end{proof}
\begin{thm}
\label{thm:PQRS_12}Suppose that $u\neq v$ in \eqref{eq:Phi(x)+u,Phi(y)+v}
and there are four intersections at $\bm{P}=(P_{1},P_{2}),\;\bm{R}=(R_{1},R_{2}),\;\bm{S}=(S_{1},S_{2}),$
and $\bm{Q}=(Q_{1},Q_{2})$, with $P_{1}\leq R_{1}\leq Q_{1}\leq S_{1}.$
Then, for a fixed $v$,
\begin{enumerate}
\item $P_{2}+R_{2}+S_{2}$ increases with respect to $u$.
\item $R_{1}+S_{1}+Q_{1}$ increases with respect to $u$.
\item $P_{1}+S_{1}+Q_{1}$ increases with respect to $u$.
\item $P_{2}+R_{2}+Q_{2}>P_{2}+R_{2}+S_{2}$.
\end{enumerate}
\end{thm}

\begin{proof}
Without loss of generality, we may assume that $u>v$. By setting
$u:=v+w$, for $w>0$, the equations in \eqref{eq:Phi(x)+u,Phi(y)+v}
become
\[
y=\Phi(x)+v+w\;\;\;\text{and}\;\;\;x=\Phi(y)+v,
\]
where $v,w>0$. As in the proof of Theorem \ref{thm:P+R+S thm}, by
scaling and translating the function $\Phi$, it suffices to consider
\[
y=\widetilde{\Phi}(x)+\frac{v+w}{J}\;\;\;\text{and}\;\;\;x=\widetilde{\Phi}(y)+\frac{v}{J},
\]
where $\widetilde{\Phi}$ was defined in \eqref{eq:tilde f(x)}. We
will use the notation $\widetilde{P}_{i}=-1+\frac{1}{P_{i}}$ for
$i=1,2$, and the same applies to the other functions.

(1) Since $R_{2}$ is increasing with respect to $w$, it suffices
to show that $P_{2}+S_{2}$ is increasing with respect to $w$. By
the definitions of $P_{1},P_{2},S_{1}$, and $S_{2}$, we have
\begin{align}
 & P_{2}=\widetilde{\Phi}(P_{1})+\frac{v+w}{J}\;\;\;\text{and}\;\;\;P_{1}=\widetilde{\Phi}(P_{2})+\frac{v}{J},\label{eq:P_1, P_2 def}\\
 & S_{2}=\widetilde{\Phi}(S_{1})+\frac{v+w}{J}\;\;\;\text{and}\;\;\;S_{1}=\widetilde{\Phi}(S_{2})+\frac{v}{J}.\label{eq:S_1 S_2}
\end{align}
Differentiating \eqref{eq:P_1, P_2 def} with respect to $w$, we
obtain
\begin{equation}
\frac{\partial P_{1}}{\partial w}=\frac{\widetilde{P}_{2}}{J^{2}-\widetilde{P}_{1}\widetilde{P}_{2}},\;\;\;\text{and}\;\;\;\frac{\partial P_{2}}{\partial w}=\frac{1}{J^{2}-\widetilde{P}_{1}\widetilde{P}_{2}}.\label{eq:1st deri of a}
\end{equation}
Similar equations can be obtained for functions $S_{1}$ and $S_{2}$
by differentiating \eqref{eq:S_1 S_2}. Since $P_{2}$ is decreasing
and $S_{2}$ is increasing with respect to $w$, we have $J^{2}-\widetilde{P}_{1}\widetilde{P}_{2}<0$
and $J^{2}-\widetilde{S}_{1}\widetilde{S}_{2}>0$. Thus, we have to
show that
\[
\frac{\partial P_{2}}{\partial w}+\frac{\partial S_{2}}{\partial w}\geq0,\;\;\text{or equivalently,}\;\;\widetilde{P}_{1}\widetilde{P}_{2}+\widetilde{S}_{1}\widetilde{S}_{2}-2J^{2}\geq0.
\]
From \eqref{eq:P_1, P_2 def} and \eqref{eq:S_1 S_2}, we obtain the
following:
\begin{align*}
 & J(S_{2}-P_{2})=-S_{1}+\log S_{1}+P_{1}-\log P_{1}=\int_{P_{1}}^{S_{1}}\left(-1+\frac{1}{t}\right)dt,\\
 & J(S_{1}-P_{1})=-S_{2}+\log S_{2}+P_{2}-\log P_{2}=\int_{P_{2}}^{S_{2}}\left(-1+\frac{1}{t}\right)dt.
\end{align*}
Since the area of the trapezoid with vertices $(P_{1},0),(S_{1},0),\big(P_{1},\frac{1}{P_{1}}\big),$
and $\big(S_{1},\frac{1}{S_{1}}\big)$ is greater than the definite
integral $\int_{P_{1}}^{S_{1}}\frac{1}{t}dt$, we have
\begin{align*}
\int_{P_{1}}^{S_{1}}\left(-1+\frac{1}{t}\right)dt & \leq\frac{1}{2}\left(\frac{1}{P_{1}}+\frac{1}{S_{1}}\right)(S_{1}-P_{1})-(S_{1}-P_{1})=\frac{1}{2}(\widetilde{P}_{1}+\widetilde{S}_{1})(S_{1}-P_{1}).
\end{align*}
Thus, we obtain
\begin{equation}
J(S_{2}-P_{2})\leq\frac{1}{2}(\widetilde{P}_{1}+\widetilde{S}_{1})(S_{1}-P_{1}).\label{eq:A1C1}
\end{equation}
Similarly, we have
\begin{equation}
J(S_{1}-P_{1})\leq\frac{1}{2}(\widetilde{P}_{2}+\widetilde{S}_{2})(S_{2}-P_{2}).\label{eq:A2C2}
\end{equation}
Multiplying \eqref{eq:A1C1} and \eqref{eq:A2C2}, we obtain the inequality
\begin{align}
J^{2} & \leq\frac{1}{4}(\widetilde{P}_{1}+\widetilde{S}_{1})(\widetilde{P}_{2}+\widetilde{S}_{2}).\label{eq:JAC}
\end{align}
Using \eqref{eq:JAC}, we have the following estimation:
\begin{align*}
2(\widetilde{P}_{1}\widetilde{P}_{2}+\widetilde{S}_{1}\widetilde{S}_{2}-2J^{2}) & =2(\widetilde{P}_{1}\widetilde{P}_{2}+\widetilde{S}_{1}\widetilde{S}_{2}-2J^{2})-(\widetilde{P}_{1}+\widetilde{S}_{1})(\widetilde{P}_{2}+\widetilde{S}_{2})+(\widetilde{P}_{1}+\widetilde{S}_{1})(\widetilde{P}_{2}+\widetilde{S}_{2})\\
 & =(\widetilde{P}_{1}-\widetilde{S}_{1})(\widetilde{P}_{2}-\widetilde{S}_{2})+(\widetilde{P}_{1}+\widetilde{S}_{1})(\widetilde{P}_{2}+\widetilde{S}_{2})-4J^{2}\\
 & =\left(\frac{1}{P_{1}}-\frac{1}{S_{1}}\right)\left(\frac{1}{P_{2}}-\frac{1}{S_{2}}\right)+(\widetilde{P}_{1}+\widetilde{S}_{1})(\widetilde{P}_{2}+\widetilde{S}_{2})-4J^{2}\geq0,
\end{align*}
since $P_{1}<P_{2}<S_{2}<S_{1}.$ This proves (1).

(2) Since $S_{1}$ is increasing with respect to $w$, it suffices
to show that $R_{1}+Q_{1}$ is increasing with respect to $w$. Hence,
we need to show that
\begin{equation}
\frac{\partial R_{1}}{\partial w}+\frac{\partial Q_{1}}{\partial w}\geq0,\;\;\text{or equivalently,}\;\;J^{2}(\widetilde{R}_{2}+\widetilde{Q}_{2})-\widetilde{R}_{2}\widetilde{Q}_{2}(\widetilde{R}_{1}+\widetilde{Q}_{1})\leq0.\label{eq:bd1}
\end{equation}
Since $R_{2}>1$ and $Q_{2}>1$, we have $\widetilde{R}_{2}<0,\;\widetilde{Q}_{2}<0$.
Hence, if $\widetilde{R}_{1}+\widetilde{Q}_{1}\geq0$, then \eqref{eq:bd1}
holds. Now, we assume that $\widetilde{R}_{1}+\widetilde{Q}_{1}<0.$
By the same argument in the proof of (1), we have
\begin{align*}
 & J(R_{2}-Q_{2})=\int_{Q_{1}}^{R_{1}}\left(-1+\frac{1}{t}\right)dt\geq-\frac{1}{2}(\widetilde{R}_{1}+\widetilde{Q}_{1})(Q_{1}-R_{1}),\\
 & J(Q_{1}-R_{1})=\int_{R_{2}}^{Q_{2}}\left(-1+\frac{1}{t}\right)dt\geq-\frac{1}{2}(\widetilde{R}_{2}+\widetilde{Q}_{2})(R_{2}-Q_{2}).
\end{align*}
Since $Q_{1}-R_{1}>0$ and $R_{2}-Q_{2}>0$, multiplying the aforementioned
inequalities, we obtain
\begin{equation}
J^{2}\geq\frac{1}{4}(\widetilde{R}_{1}+\widetilde{Q}_{1})(\widetilde{R}_{2}+\widetilde{Q}_{2}).\label{eq:R_1 S_1 Q_1}
\end{equation}
By \eqref{eq:R_1 S_1 Q_1}, an elementary computation yields
\[
J^{2}\left(\widetilde{R}_{2}+\widetilde{Q}_{2}\right)-\widetilde{R}_{2}\widetilde{Q}_{2}(\widetilde{R}_{1}+\widetilde{Q}_{1})\leq\frac{1}{4}(\widetilde{R}_{2}-\widetilde{Q}_{2})^{2}(\widetilde{R}_{1}+\widetilde{Q}_{1})\leq0.
\]
This proves (2).

(3) It suffices to verify that $P_{1}+Q_{1}$ is increasing with respect
to $u$. Thus, we have to show that
\begin{equation}
\frac{\partial P_{1}}{\partial w}+\frac{\partial Q_{1}}{\partial w}\geq0\;\;\text{or equivalently,}\;\;J^{2}(\widetilde{P}_{2}+\widetilde{Q}_{2})-\widetilde{P}_{2}\widetilde{Q}_{2}(\widetilde{P}_{1}+\widetilde{Q}_{1})\geq0.\label{eq:a_1+d_1 deriv}
\end{equation}
Since $P_{1}<P_{2}<1<Q_{2}<Q_{1}$, we have $\widetilde{P}_{1},\widetilde{P}_{2}>0$
and $\widetilde{Q}_{1},\widetilde{Q}_{2}<0$. By the same argument
in the proof of (2), we obtain
\begin{align*}
 & J(Q_{2}-P_{2})=\int_{P_{1}}^{Q_{1}}\left(-1+\frac{1}{t}\right)dt\leq\frac{1}{2}(\widetilde{P}_{1}+\widetilde{Q}_{1})(Q_{1}-P_{1}),\\
 & J(Q_{1}-P_{1})=\int_{P_{2}}^{Q_{2}}\left(-1+\frac{1}{t}\right)dt\leq\frac{1}{2}(\widetilde{P}_{2}+\widetilde{Q}_{2})(Q_{2}-P_{2}).
\end{align*}
Each left-hand side of the aforementioned equations is positive; thus,
we obtain $\widetilde{P}_{1}+\widetilde{Q}_{1}>0$ and $\widetilde{P}_{2}+\widetilde{Q}_{2}>0$.
Therefore, \eqref{eq:a_1+d_1 deriv} holds. This proves (3).

(4) This is obviously true. This completes the proof.
\end{proof}

\subsection{\label{subsec:Middle temp}General case: middle-temperature regime
$(\beta_{1}<\beta<\beta_{3})$}

In this section, we investigate all the local minima and lowest saddles
of $F_{\beta,\;J}$ belonging to either $\mathfrak{S}^{2}$ or $\mathfrak{L}^{2}$
for the middle-temperature and compare the values of the local minima
according to the temperature. Furthermore, we specify the synchronization
boundary and the desynchronization boundary, respectively. Recall
the definition of the function $\psi_{2}$ is in \eqref{eq:psi_1,2,3}.
Then, the function $\psi_{2}$ has the following properties.
\begin{prop}
The function $\psi_{2}$ is a continuous function such that 
\[
0<\psi_{2}(\beta)<\frac{1}{10}\;\;\text{and}\;\;\;\psi_{2}(\beta_{1})=\psi_{2}(\beta_{3})=0.
\]
\end{prop}

\begin{proof}
Note that $x_{l}$ is a continuous function of $\beta$ and $x_{l}\in\left(m_{1},\frac{1}{3}\right)$
if $\beta_{1}<\beta<\beta_{3}$. Since $\psi_{2}$ is a composition
of continuous functions, it is continuous. Since $m_{1}<x_{l}<\frac{1}{3}$,
we have
\[
\beta-\frac{1}{3x_{l}(1-2x_{l})}=\frac{1}{1-3x_{l}}\log\frac{1-2x_{l}}{x_{l}}-\frac{1}{3x_{l}(1-2x_{l})}>0.
\]
Thus, $\psi_{2}(\beta)>0$ for $\beta_{1}<\beta<\beta_{3}$. By the
definition of $\beta,$ the inequality $\psi_{2}(\beta)<\frac{1}{10}$
is equivalent to
\begin{equation}
\frac{9}{1-3x_{l}}\log\frac{1-2x_{l}}{x_{l}}<\frac{11}{3x_{l}(1-2x_{l})}.\label{eq:ineq psi <1/10}
\end{equation}
An elementary calculation shows that \eqref{eq:ineq psi <1/10} holds
when $0<x_{l}<\frac{1}{3}$; hence, $\psi_{2}(\beta)<\frac{1}{10}$.
From the numerator of \eqref{eq:derivative invtem}, $m_{1}$ satisfies
the equation
\begin{equation}
1-3m_{1}+3m_{1}(1-2m_{1})\log\frac{m_{1}}{1-2m_{1}}=0.\label{eq:m_1}
\end{equation}
By the definition of $\beta_{1}$, \ref{eq:m_1} becomes
\[
\beta_{1}=\frac{1}{1-3m_{1}}\log\frac{1-2m_{1}}{m_{1}}=\frac{1}{3m_{1}(1-2m_{1})}.
\]
Since $\lim_{\beta\rightarrow\beta_{1}}x_{l}=m_{1}$, we have $\psi_{2}(\beta_{1})=\lim_{\beta\rightarrow\beta_{1}}\psi_{2}(\beta)=0$.
In addition, $\lim_{\beta\rightarrow\beta_{3}}x_{l}=\frac{1}{3}$
implies that $\psi_{2}(\beta_{3})=\lim_{\beta\rightarrow\beta_{3}}\psi_{2}(\beta)=0$.
This completes the proof.
\end{proof}
\begin{thm}
\label{thm:mid temp psi_2}Suppose that $\beta_{1}<\beta<\beta_{3}$.
Then, we have the following results.
\begin{enumerate}
\item If $J>\psi_{2}(\beta)$, then there is a local minimum at $\bm{q_{3}}$,
three local minima at $\bm{x_{s}}\in\mathfrak{S}^{2}$ and three saddle
points at $\bm{x_{l}}\in\mathfrak{L}^{2}$.
\item If $J<\psi_{2}(\beta)$, then there is a local minimum at $\bm{q_{3}}$,
three local minima at $\bm{x_{s}}\in\mathfrak{S}^{2}$ and six saddle
points of the form $(s,s,1-2s,t,t,1-2t)$ and $(t,t,1-2t,s,s,1-2s)$
with $s<t$ up to permutations
\item (Comparison of local minima) If $\beta_{1}<\beta<\beta_{2}$, then
$F_{\beta,\;J}(\bm{q_{3}})<F_{\beta,\;J}(\bm{x_{s}})$, and if $\beta_{2}<\beta<\beta_{3}$,
then $F_{\beta,\;J}(\bm{q_{3}})>F_{\beta,\;J}(\bm{x_{s}})$.
\end{enumerate}
\end{thm}

\begin{proof}
Note that the assumption $\beta_{1}<\beta<\beta_{3}$ implies that
$\Psi'(x_{l})<1$ and $x_{s}<x_{l}<\frac{1}{3}<\frac{1}{\beta}$.

(1) If $J>\psi_{2}(\beta)$, which is equivalent to $\Psi'(x_{l})>-1$,
then the function $y=\Psi(x)$ intersects only at $(x_{s},x_{s}),(x_{l},x_{l})$,
and $(\frac{1}{3},\frac{1}{3})$ with $x=\Psi(y)$. By observing the
graphs in \eqref{eq:y=00003DPsi(x), x=00003DPsi(y)}, it follows that
$\Psi'(\frac{1}{3})>1$, and $\Psi'(x_{s})>1$. Moreover, we know
that $-1<\Psi'(x_{l})<1$. Since $\Phi'\left(\frac{1}{\beta}\right)=1$
and since $x_{s}<x_{l}<\frac{1}{3}<\frac{1}{\beta}$, it follows that
$\Phi'(\frac{1}{3})>1$, $\Phi'(x_{s})>1$, and $\Phi'(x_{l})>1$.
Therefore, $F_{\beta,\;J}$ has local minima at $\bm{q_{3}}$ and
$\bm{x_{s}}\in\mathfrak{S}^{2}$, and it has saddle points at $\bm{x_{l}}\in\mathfrak{L}^{2}$
by Lemma \ref{lem:eigenvalues}. This proves (1).

(2) If $J<\psi_{2}(\beta)$, which is equivalent to $\Psi'(x_{l})<-1$,
then $y=\Psi(x)$ and $x=\Psi(y)$ intersect not only at the original
three points, but also at $(s,t)$ and $(t,s)$ with $s<x_{l}<t<\frac{1}{3}$.
Since $\Psi'(x_{l})<-1$, $\bm{x_{l}}$ is no longer a saddle point
of $F_{\beta,\;J}$ by Lemma \ref{lem:eigenvalues}. We claim that
$(s,s,1-2s,t,t,1-2t)$ and $(t,t,1-2t,s,s,1-2s)$ are six new saddle
points of $F_{\beta,\;J}$ up to permutations. Since $s<\frac{1}{\beta}$
and $t<\frac{1}{\beta}$, we have $\Phi'(s)>1$ and $\Phi'(t)>1$.
By the inverse function theorem, we obtain $\Psi'(s)\Psi'(t)<1$.
Therefore, $F_{\beta,\;J}$ has saddle points at $(s,s,1-2s,t,t,1-2t)$
and $(t,t,1-2t,s,s,1-2s)$ by Lemma \ref{lem:eigenvalues}. The arguments
for $\bm{q_{3}}$ and $\bm{x_{s}}$ are the same as (1). This proves
(2).

(3) A straightforward computation yields
\[
F_{\beta,\;J}(\bm{x_{s}})-F_{\beta,\;J}(\bm{q_{3}})=(1+J)\widetilde{F}(x_{s}),
\]
where $\widetilde{F}(x)$ was defined in \eqref{eq:F tilde}. The
rest of the argument is the same as the proof of Theorem \ref{thm:J=00003Dinfty}.
This completes the proof.
\end{proof}
Recall that we defined the function $\psi_{3}$ in \eqref{eq:psi_1,2,3}.
Then, we have the following theorem.
\begin{thm}
\label{thm:psi_3 boundary}Under the assumptions in Theorem \ref{thm:P+R+S thm},
if $J>\psi_{3}(\beta)$, then the function $(P+R+S)(u)>1$ for all
$u>0$.
\end{thm}

\begin{proof}
By Theorem \ref{thm:P+R+S thm}, we know that $(P+R+S)(u)$ is convex;
thus it has a unique minimum. We claim that the minimum value is greater
than one when $J>\psi_{3}(\beta)$. For convenience, we set $\widetilde{P}_{\beta,J}:=-1+\frac{1+J}{\beta P},$
and the same applies to the other functions. A straightforward calculation
shows that
\[
P'(u)=\frac{J}{J-\widetilde{P}_{\beta,J}},\;\;\;R'(u)=\frac{J(J+\widetilde{S}_{\beta,J})}{J^{2}-\widetilde{R}_{\beta,J}\widetilde{S}_{\beta,J}},\;\;\;\text{and}\;\;\;S'(u)=\frac{J(J+\widetilde{R}_{\beta,J})}{J^{2}-\widetilde{R}_{\beta,J}\widetilde{S}_{\beta,J}};
\]
hence,
\[
(P+R+S)'(u)=0\;\;\text{is equivalent to}\;\;(\widetilde{P}_{\beta,J}-J)(\widetilde{R}_{\beta,J}+\widetilde{S}_{\beta,J}+2J)-(J^{2}-\widetilde{R}_{\beta,J}\widetilde{S}_{\beta,J})=0.
\]
 Applying the method of Lagrange multipliers to 
\begin{align*}
G(P,R,S) & =P+R+S,\\
H(P,R,S) & =(\widetilde{P}_{\beta,J}-J)(\widetilde{R}_{\beta,J}+\widetilde{S}_{\beta,J}+2J)-(J^{2}-\widetilde{R}_{\beta,J}\widetilde{S}_{\beta,J})=0,
\end{align*}
yields the following equations:
\begin{equation}
\frac{1}{P^{2}}(2J+\widetilde{R}_{\beta,J}+\widetilde{S}_{\beta,J})=\frac{1}{R^{2}}(\widetilde{P}_{\beta,J}+\widetilde{S}_{\beta,J}-J)=\frac{1}{S^{2}}(\widetilde{P}_{\beta,J}+\widetilde{R}_{\beta,J}-J).\label{eq:Lagrange multiplier}
\end{equation}
Calculating the equation of the second and the third in \eqref{eq:Lagrange multiplier},
we obtain
\begin{equation}
R+S=\frac{\frac{1+J}{\beta}}{J+2-\frac{1+J}{\beta P}}.\label{eq:R+S}
\end{equation}
Therefore, $G$ becomes a function of $P$:
\begin{equation}
G(P,R,S)=P+R+S=\frac{(J+2)\frac{1+J}{\beta}}{\left(J+2-\frac{1+J}{\beta P}\right)\frac{1+J}{\beta P}}.\label{eq:f(a,b,c)}
\end{equation}
The denominator of \eqref{eq:f(a,b,c)} is a quadratic function of
$\frac{1+J}{\beta P}$; hence it has a maximum value at $\frac{1+J}{\beta P}=\frac{J+2}{2}$.
However, if we apply the argument in the proof of (2) in Lemma \ref{lem:P+R+S }
without scaling and translating $\Phi(x)$, then we obtain $\widetilde{P}_{\beta,J}>3J$.
This implies that $\frac{1+J}{\beta P}\geq3J+1$. Since $3J+1\geq\frac{J+2}{2}$,
\eqref{eq:f(a,b,c)} has a minimum at $\frac{1+J}{\beta P}=3J+1$,
and its value of $G$ is $\frac{(1+J)(2+J)}{\beta(1-2J)(1+3J)}$.
Therefore, if 
\begin{equation}
\frac{(1+J)(2+J)}{\beta(1-2J)(1+3J)}>1,\label{eq:varphi2}
\end{equation}
then $(P+R+S)(u)>1$ for all $u$. Solving the inequality \eqref{eq:varphi2}
for $J$, we obtain $J>\psi_{3}(\beta)$. This completes the proof.
\end{proof}

\subsection{\label{subsec:High temp}General case: high-temperature regime $(\beta<\beta_{1})$}

In this section, we show that the two components are synchronized
in high-temperature by proving that there is only one local minimum
at $\bm{q_{3}}$.
\begin{thm}
\label{thm:high temp}If $0<\beta<\beta_{1}$, then there is only
one global minimum at $\bm{q_{3}}$.
\end{thm}

\begin{proof}
The fact that $\bm{q_{3}}$ is a local minimum is straightforward
by Lemma \ref{lem:eigenvalues}. Since $\beta<\beta_{1}$, by the
definitions of $\mathfrak{S}^{2}$ and $\mathfrak{L}^{2}$, there
is no critical point belonging to either $\mathfrak{S}^{2}$ or $\mathfrak{L}^{2}$.
We claim that there is no other kind of critical points other than
$\bm{q_{3}}$. First, we assume that $u=v$ in \eqref{eq:Phi(x)+u,Phi(y)+v}.
When \eqref{eq:Phi(x)+u,Phi(y)+v} has only one solution at $x=\frac{1}{\beta},$
the value $\frac{1}{\beta}>\frac{1}{3}$. After increasing $u$ minutely,
which allows \eqref{eq:Phi(x)+u,Phi(y)+v} to have two intersections
at $\bm{P}=(P,P)$ and $\bm{Q}=(Q,Q)$, by Lemma \ref{lem:2P+Q increase},
$2P+Q$ increases with respect to $u$. We increase $u$ a little
more so that \eqref{eq:Phi(x)+u,Phi(y)+v} have four intersections
at $\bm{P}=(P,P),\;\bm{R}=(R,S),\;\bm{S}=(S,R)$, and $\bm{Q}=(Q,Q)$
with $P<R<Q<S$. In this case, $P+R+S>2P+Q>\frac{3}{\beta}>1$. Moreover,
$R+S+Q>1$. That is, there is no critical point of the form $(P,R,S,P,S,R)$
or $(Q,R,S,Q,S,R)$ up to permutations. Now, without loss of generality,
we may assume that $u>v$. In this case, by Theorem \ref{thm:PQRS_12},
each sum of the coordinates of all possible combinations of the three
intersections in \eqref{eq:Phi(x)+u,Phi(y)+v} is greater than one.
This completes the proof.
\end{proof}
Finally, we prove the main theorem which is Theorem \ref{thm:sync-desync phase}.
\begin{proof}[Proof of Theorem \ref{thm:sync-desync phase}]
$ $

We prove the synchronization part first. By Theorem \ref{thm:high temp},
the two components are synchronized in high-temperature $(\beta<\beta_{1})$.
Proposition \ref{prop:psi_1 boundary}, and Theorems \ref{thm:P+R+S thm}
and \ref{thm:PQRS_12} imply that if $J>\max(\psi_{1}(\beta),J_{c})$,
then there are no local minima and lowest saddles of $F_{\beta,\;J}$
that does not belong to either $\mathfrak{S}^{2}$ or $\mathfrak{L}^{2}$,
that is, the two components are synchronized. In addition, Theorem
\ref{thm:PQRS_12} and \ref{thm:psi_3 boundary} imply that if $J>\psi_{3}(\beta)$,
then there are no local minima and lowest saddles of $F_{\beta,\;J}$
other than $\mathfrak{S}^{2}$ or $\mathfrak{L}^{2}$. Therefore,
the two components are synchronized when $J>\psi_{s}(\beta)$. However,
if $J<\psi_{d}(\beta)$, then by the second statement of Theorems
\ref{thm:low temp psi_1} and \ref{thm:mid temp psi_2}, the two components
are desynchronized. This completes the proof.
\end{proof}

\appendix

\section{\label{sec:app A}Proof of proposition \ref{Prop Ham}}

Here, we present a proof of Proposition \ref{Prop Ham}.
\begin{proof}[Proof of Proposition \ref{Prop Ham}]
 $ $

For $\bm{x}\in\Xi_{N}^{2}$, from the definition of $\mathbb{H}_{N}$,
\begin{align*}
 & \sum_{\bm{\sigma}\in\Omega:\bm{r}(\bm{\sigma})=\bm{x}}\frac{1}{Z_{N,\;\beta}}e^{-\beta\mathbb{H}_{N}(\bm{\sigma})},\\
 & =\frac{N!}{(Nx_{1}^{(1)})!\;\cdots\;(Nx_{q}^{(1)})!}\cdot\text{\ensuremath{\frac{N!}{(Nx_{1}^{(2)})!\;\cdots\;(Nx_{q}^{(2)})!}}}\\
 & \times\frac{1}{Z_{N,\;\beta}}\exp\left[\text{\ensuremath{\frac{\beta}{N(1+J)}}}\left\{ \sum_{k=1,\;2}\sum_{i=1}^{q}\frac{1}{2}\left\{ (Nx_{i}^{(k)})(Nx_{i}^{(k)}-1)\right\} +J\sum_{i=1}^{q}Nx_{i}^{(1)}Nx_{i}^{(2)}\right\} \right],
\end{align*}
and by Stirling's formula, we have
\begin{align*}
 & \approx\frac{\exp\left\{ -\beta/(1+J)\right\} }{(\sqrt{2\pi N})^{2(q-1)}\sqrt{\prod_{k=1,\;2}\prod_{i=1}^{q}x_{i}^{(k)}}Z_{N,\;\beta}}\\
 & \times\exp\left[N\frac{\beta}{1+J}\left\{ \sum_{k=1,\;2}\sum_{i=1}^{q}\frac{1}{2}(x_{i}^{(k)})^{2}+J\sum_{i=1}^{q}x_{i}^{(1)}x_{i}^{(2)}-\frac{1+J}{\beta}\left(\sum_{k=1,\;2}\sum_{i=1}^{q}x_{i}^{(k)}\log x_{i}^{(k)}\right)\right\} \right],\\
 & =\frac{1}{\widehat{Z}_{N,\;\beta,\;J}}\exp\left\{ -N\frac{\beta}{1+J}\left(F_{\beta,\;J}(\bm{x})+\frac{1}{N}G_{N,\,\beta}(\boldsymbol{x})\right)\right\} ,
\end{align*}
where
\begin{align*}
F_{\beta,\;J}(\bm{x}) & =-\sum_{k=1,\;2}\sum_{i=1}^{q}\frac{1}{2}(x_{i}^{(k)})^{2}-J\sum_{i=1}^{q}x_{i}^{(1)}x_{i}^{(2)}+\frac{1+J}{\beta}\left(\sum_{k=1,\;2}\sum_{i=1}^{q}x_{i}^{(k)}\log x_{i}^{(k)}\right),\\
G_{N,\;\beta,\;J}(\bm{x}) & =\frac{1+J}{2\beta}\log\left(\prod_{k=1,\;2}\prod_{j=1}^{q}x_{i}^{(k)}\right)+O\left(N^{-(q-1)}\right).
\end{align*}
We decompose the function $F_{\beta,\;J}$ into an energy part $H$
and entropy part $S$. That is,
\[
F_{\beta,\;J}(\bm{x})=H(\bm{x})+\frac{1+J}{\beta}S(\bm{x}),
\]
where
\[
H(\bm{x})=-\sum_{k=1,\;2}\sum_{i=1}^{q}\frac{1}{2}(x_{i}^{(k)})^{2}-J\sum_{i=1}^{q}x_{i}^{(1)}x_{i}^{(2)}\;\;\;\text{and}\;\;\;S(\bm{x})=\sum_{k=1,\;2}\sum_{i=1}^{q}x_{i}^{(k)}\log x_{i}^{(k)}.
\]
\end{proof}

\end{document}